\documentclass[11pt]{article}
\usepackage[english]{babel}
\usepackage[a4paper, left=3.4cm, right=3.4cm, top=3cm]{geometry}

\usepackage{xfrac}

\usepackage{mathptmx}      

\usepackage{amsmath,amsthm}
\usepackage{amssymb}
\usepackage{amsfonts}
\usepackage{oldgerm}
\usepackage{mathrsfs}
\usepackage{comment}
\usepackage{tikz}
\usepackage{pgfplots}
\usepackage[scanall]{psfrag}
\usepackage{graphicx,pst-all,bm}
\usepackage{flushend}
\usepackage{subfig}
\usepackage{mdwlist}
\usepackage{multirow}
\usepackage{color}

\usepackage{longtable}
\usepackage{cuted}
\usepackage{verbatim}

\usepackage{siunitx}

\usepackage{blkarray}
\usepackage{booktabs}

\newcommand{\N}{\mathbb{N}}

\newcommand{\reals}{\mathbb{R}}

\DeclareMathOperator{\im}{im \, }

\newcommand{\pPi}{\mathnormal{\Pi}} 
\newcommand{\rank}{\operatorname{rank}}
\newcommand{\SUT}{SU\!T}
\newcommand{\BUT}{BU\!T}

\newtheorem{proposition}{Proposition}[section]
\newtheorem{theorem}[proposition]{Theorem}
\newtheorem{corollary}[proposition]{Corollary}
\newtheorem{lemma}[proposition]{Lemma}
\newtheorem{definition}[proposition]{Definition}
\newtheorem{remark}[proposition]{Remark}
\newtheorem{example}[proposition]{Example}

\pgfplotsset{compat=1.17}
\usepackage{authblk}

\begin{document}

\title{Computing standard canonical forms of regular linear time varying DAEs via a preliminary stage}

\author[1]{Diana Est\'evez Schwarz}
\author[2]{Ren\'e Lamour}
\author[2]{Roswitha M\"arz}

\affil[1]{\small Berliner Hochschule f\"ur Technik}
\affil[2]{\small Humboldt University of Berlin, Institute of Mathematics}

\maketitle

\begin{abstract}

For regular linear time-invariant DAEs  the corresponding matrix pencil is regular and the computation of a standard canonical form is well-understood. 
Although the investigation of linear  DAEs with time-varying coefficients is more complex, it is analogously related to the examination of pairs of time-dependent matrix functions.

We show how the computation of a standard canonical form (SCF) for linear time-varying  DAEs becomes possible if a suitable block structure of the pair of matrix functions is found in a preliminary stage. 
Starting from this preliminary stage, an iterative process delivers an SCF. This iteration terminates in finitely many steps due to the nilpotency of an involved matrix. The corresponding transformation matrix functions can be provided systematically, which leads also to new representations of the canonical subspaces and projectors related to the original DAE.

We demonstrate how the structured canonical forms resulting in the tractability and strangeness frameworks and some DAEs in Hessenberg form from applications can be transformed into this preliminary stage and discuss some examples. 

\end{abstract}

\setcounter{secnumdepth}{3}
\setcounter{tocdepth}{3}

\tableofcontents
\pagebreak

\section{Introduction} \label{sec:Intro}

We focus on linear differential algebraic equations (DAEs) in standard form,
\begin{align}\label{eq:1.DAE}
 Ex'+Fx=q,
\end{align}
in which $E,F:\mathcal I\rightarrow \reals^{m\times m}$ are sufficiently smooth, at least continuous, matrix functions on the interval $\mathcal I\subseteq\reals$. We assume regularity in the sense of \cite{commonground2024}, i.e. in particular $\im \begin{bmatrix}
	E(t) & F(t)
\end{bmatrix} = \reals^m$  and a constant $r=\rank E(t)<m$ for all $t\in\mathcal I$ . 
Differential and algebraic parts may be intertwined in a complex manner, which immediately leads to the question of whether there is an equivalent formulation with separated parts. 

\bigskip

Let us first recall  what equivalence means, see e.g., \cite{CRR}, \cite{KuMe2024}, among others.
The DAE \eqref{eq:1.DAE} and the DAE
\begin{align}\label{eq:1.DAE_bar}
 \bar E  \bar x'+\bar F \bar x=\bar q,
\end{align}
and correspondingly the two pairs of matrix functions $\{E,F\}$ and $\{\bar E,\bar F\}$, are called \textit{equivalent}, if there exist pointwise nonsingular, sufficiently smooth matrix functions $L, K:\mathcal I\rightarrow \reals^{m\times m}$, such that
\begin{align}\label{1.Equivalence}
 \bar E=LEK,\quad \bar F=LFK+LEK'.
\end{align}
For equivalent matrix pairs we use the notation
$
\{E,F\} \ \sim \ \{\bar E,\bar F\}
$  (or  $
\{E,F\} \ \overset{L,K}{\sim} \ \{\bar E,\bar F\}
$ if we want to emphasize the used nonsingular transfomration matrices),
i.e., by definition it holds
\begin{align*}
\{E,F\} \quad \sim \quad  \{LEK \ , \ LFK+LEK'\}.
\end{align*}

Recall that these transformations are reflexive, symmetric and transitive. Moreover, for the DAE they correspond to
\begin{itemize}
	\item the premultiplication of \eqref{eq:1.DAE} by $L$ and
	\item and the transformation of the unknown function $x=K\bar x$ 
\end{itemize}
leading to 
\[
LE(K\bar x)'+LF(K \bar x)=Lq =: \bar q,
\]
that with \eqref{1.Equivalence} corresponds to the DAE \eqref{eq:1.DAE_bar}. 
\bigskip

The question of particularly favorable equivalent representations, that is the transformability of DAEs into canonical forms, has been investigated for more than forty
 years.
A pair of the form
\begin{eqnarray*}
	\left\{ \begin{bmatrix}
	I_d & 0 \\
	0 & N
\end{bmatrix}, \begin{bmatrix}
	\Omega & 0 \\
	0 & I_{a}
\end{bmatrix} \right\} , \quad \Omega: \mathcal I\rightarrow\reals^{d\times d}, \quad N : \mathcal I\rightarrow\reals^{a\times a}, \quad d,a \in \N, \quad m=d+a,
\end{eqnarray*}
where $N$ is pointwise nilpotent and lower or upper triangular, is said to be in \emph{standard canonical form} (SCF), cf.\ \cite{CaPe83}.  If in addition $N$ is constant, then the system is in \emph{strong standard canonical form} (SSCF).

\bigskip

Starting from results from \cite{HaMae2023}, in \cite{commonground2024} the equivalence of different characterizations of regular $\{E, F \}$ is given. 
The common ground of these 
statements is the existence of  the index $\mu$ and the canonical characteristics
\begin{align}\label{char}
 r<m,\; \theta_{0}\geq\cdots\geq\theta_{\mu-2}>\theta_{\mu-1}=0,\; d=r-\sum_{i=0}^{\mu-2}\theta_i,
\end{align}
that are constant, persist under equivalence transformations and can be computed using different index frameworks, cf.\ \cite{commonground2024}.
\medskip

Only recently, in \cite{SSCF25} it was proved that regularity also means equivalent transformability into a strong standard canonical form, with a constant, nilpotent matrix $N$  of index $\mu$ and
\[
\theta_i=\rank N^{i+1} - \rank N^{i+2}, \quad i=0, \ldots, \mu-2, \quad d=r-\sum_{i=0}^{\mu-2} \theta_i,
\]
that is
\begin{eqnarray*}
\rank N=r-d, \quad \rank N^i = r-d - \sum_{k=0}^{i-2} \theta_k, \quad i=2, \ldots, \mu.
\end{eqnarray*}
In particular, this constant matrix can be assumed to have a certain elementary block-structure $N^{(E_{c/r})}$ that is determined by the canonical characteristics only, such that for a regular pair we know
\begin{eqnarray*}
\{E, F\} \sim
	\left\{ \begin{bmatrix}
	I_d & 0 \\
	0 & N^{(E_{c/r})}
\end{bmatrix}, \begin{bmatrix}
	\Omega & 0 \\
	0 & I_{a}
\end{bmatrix} \right\} , \quad \Omega: \mathcal I\rightarrow\reals^{d\times d}, \quad N^{(E_{c/r})}  \in \reals^{a\times a}.
\end{eqnarray*}

 This will be the starting point here, where we suppose that, by an initial equivalence transformation (a kind of preprocessing)
\begin{eqnarray}
\{E, F\} \overset{L_0,K_0}{\sim} \left\{ \begin{bmatrix}
	I_d & 0\\
	0 & N^{(E_{c/r})}
\end{bmatrix},F_0 \right\}, 
\label{eq:Initial_eq}
\end{eqnarray}
is already given, along with certain assumptions on the block structure of $F_0$. Then a direct computation of an SSCF becomes possible in four main steps with $L_i, K_i$, $i=1, \ldots, 4$. 
Consequently,   the tranformation $K:=K_0K_1K_2K_3K_4$ permits a description of 
the so-called canonical spaces and the canonical projector.
\medskip

The paper is organized as follows.
 In  Section \ref{sec:Can_sub_forms} we briefly collect some basics from literature, 
\begin{itemize}
	\item on the one hand,  the canonical spaces and the canonical projectors related to the DAE,
	\item on the other hand,  the canonical forms resulting in the frameworks of the tractability and strangeness index. 
\end{itemize}	
	We also define a new \emph{quasi} SCF for later considerations.
\medskip

These canonical forms motivate the particular assumption for the block structure we use in Section \ref{sec:preliminary}, where the preliminary stage of the SCF is introduced.
In Section \ref{sec:Procedure} then a procedure is described that, starting from such a block-structured preliminary stage, computes an SSCF step-by-step. 
Through a more detailed examination of the transformation matrices used in the steps, a description of the canonical spaces and the canonical projector is presented in Section \ref{sec:Pican}. 
\medskip

To emphasize the applicability of the procedure, first in Section \ref{sec:T-S-canonical} we show how the canonical forms described in Section  \ref{sec:Can_sub_forms} can equivalently be transformed into the preliminary stage of an SCF. 
Secondly,  in Section \ref{sec:Hessenberg} we show how classes of DAEs in Hessenberg form that are relevant in applications can also be transformed into such a preliminary stage.
Finally, in Section \ref{sec:Campbell-Moore} a challenging index-3 example from literature is discussed in detail.
Some technical details are collected
 in the appendix.
\medskip

The computations for all examples were carried out using the Python library SymPy for symbolic mathematics.

\subsection*{List of symbols and abbreviations}
		\begin{tabular}{ll}
ODE & ordinary differential equation\\
IVP & initial value problem\\
DAE &differential-algebraic equation\\
SCF & standard canonical form\\
SSCF & strong SCF, i.e., SCF with constant nilpotent matrix $N$\\
QuasiSCF &  particular block structure with  nilpotent matrix $N^{(E_{c/r})}$ and $R \in \BUT_{nonsingular}$\\
PreSCF & preliminary stage of a QuasiSCF (and therefore SCF and SSCF)\\
$I_n $ & identity matrix of size $n$\\
$\mathcal I \subseteq \reals$ & interval \\
$m \in \N$ & dimension of the DAE \\ 
$d \in \N$& degree of freedom, dimension of pure ODEs \\
$a=m-d \in \N$ &  dimension of pure DAEs\\
$\ker A$, $\im A$ & kernel, image of a  matrix valued function $A:\mathcal I\rightarrow \reals^{m\times m}$\\
$\{E,F\}$ & pair of quadratic matrix valued functions $E,F :\mathcal I\rightarrow \reals^{m\times m}$\\
$r \in \N$ & rank of $E$\\
$\BUT$ & set of all block upper triangular  $B: \mathcal I\rightarrow \reals^{a\times a}$\\
$\BUT_{nonsingular}$ & set of all nonsingular  $R \in \BUT$ \\
$N:\mathcal I\rightarrow \reals^{a\times a}$  &  nilpotent matrix valued function \\
$\SUT$ & set of all block strictly upper triangular $N:\mathcal I\rightarrow \reals^{a\times a}$ \\
$\SUT_{column}$ & set of all $N\in \SUT$ with full column rank blocks $N_{i,i+1}$\\
$\SUT_{row}$ &   set of all $N\in \SUT$ with full row rank blocks $N_{i,i+1}$ \\
$N^{(C)} \in \reals^{a\times a}$ & constant nilpotent matrix \\
$N^{(E_c)}\in \reals^{a\times a}$ & elementary constant nilpotent matrix from $\SUT_{column}$\\
$N^{(E_r)}\in \reals^{a\times a}$ & elementary constant nilpotent matrix from $\SUT_{row}$\\
$N^{(E_{c/r})}\in \reals^{a\times a}$ & either $N^{(E_c)}$ or $N^{(E_r)}$\\
$S_{can}$ & canonical flow-subspace of the DAE\\
$N_{can}$ & canonical complement to the flow-subspace\\
$\Pi_{can}$& 	projector  onto $S_{can}$ and along $N_{can}$\\
		\end{tabular}

\section{Canonical  subspaces and canonical forms} \label{sec:Can_sub_forms}
\subsection{The canonical subspaces of a DAE}

As pointed out in \cite{HaMae2023}, \cite{hanke_maerz2025}, there are two continuously time-varying canonical subspaces
related to a DAE \eqref{eq:1.DAE}. The first is the so-called flow subspace
\begin{eqnarray*}
S_{can}(\bar{t}) = \bigg\{ \bar{x} \in \mathbb{R}^m :\ 
&\text{there is a solution } x : (\bar{t} - \delta, \bar{t} + \delta) 
\cap \mathcal I \to \mathbb{R}^m \\
&\text{of the homogeneous DAE such that } x(\bar{t}) = \bar{x} \bigg\},
\end{eqnarray*}
which is the set of consistent values for the homogeneous DAE. Due to constraints, $\dim S_{can}(\bar{t})\leq r<m$.  
For regular DAEs, the flow subspace $S_{can}$ is varying in $\reals^m$ depending on $t$ and has constant dimension $d$. The second subspace  $N_{can}(\bar{t})$ is its so-called canonical complement in $\reals^m$, such that the initial conditions
\[
x(\bar{t}) - \bar{x} \in N_{can}(\bar{t})
\]
ensure uniquely solvable IVPs for \eqref{eq:1.DAE} without any additional conditions between $q$ and $\bar{x}$. It may also vary in $\reals^m$ depending on $t$ and has constant dimension $a$. 
In contrast to DAEs, an explicit ODE would yield $S_{can}=\reals^m$, $N_{can}=\left\{0\right\}$ for $t \in \mathcal I$. Conversely, so-called pure DAEs yield $S_{can}=\left\{0\right\}$, $N_{can}=\reals^m$ for $t \in \mathcal I$.

Below, we describe both subspaces by means of SCFs. Notice that if $\{E,F\}$ is already in SCF, then obviously
\[
S_{can}= \im \begin{bmatrix}
	I_d \\
	0
\end{bmatrix}, \quad N_{can}= \im \begin{bmatrix}
	0 \\
	I_a
\end{bmatrix}.
\]

\medskip
As mentioned before, transformability into SSCF is an equivalent criterion of regularity. 
Let us for a moment suppose that $L,K$ are nonsingular matrix function such that
\begin{eqnarray}
\{E, F\} &\overset{L,K}{\sim}&  \left\{ \begin{bmatrix}
	I_d & 0\\
	0 & N^{(C)}
\end{bmatrix},\begin{bmatrix}
	\Omega & 0\\
	0 & I_{m-d}
\end{bmatrix} \right\}, 
\label{eq:SSCF}
\end{eqnarray}
for a constant strictly upper triangular nilpotent matrix $N^{(C)} \in \reals^{a\times a}$ and a matrix function $\Omega: \mathcal I \rightarrow \reals^{d\times d}$.
For
\[
\begin{bmatrix}
	\bar q_1(t)\\
	\bar q_2(t)
\end{bmatrix}:= L(t)q(t), \quad \begin{bmatrix}
	u(t)\\
	v(t)
\end{bmatrix}:=K^{-1}(t) x(t)
\]
we see that the transformed DAE decouples into two parts.
\begin{itemize}
	\item The first part,
	\begin{align}
u' +\Omega u=\bar q_1(t), \label{eq:pureODE}
\end{align}
is an explicit  ODE living in $\reals^d$.

If at $t_0\in \mathcal I$ an initial value $u_0 \in \reals^d$ is given, $u(\cdot)$ can be computed as the unique solution of the IVP. 

\item The second part,
\[
N^{(C)}v'+v=\bar q_2(t),
\]
is a so-called pure DAE, and $v$ is uniquely determined by
\[
v(t)=\sum_{j=0}^{\mu-1}(-1)^j (N^{(C)})^j (\bar q_2)^{(j)}(t), \quad t \in \mathcal I.
\]
\end{itemize}
Consequently, all solutions of the original DAE \eqref{eq:1.DAE} are given by
\[
x(t)=K(t) \begin{bmatrix}
	u(t)\\
	v(t)
\end{bmatrix}, \quad t \in \mathcal I,
\]
and, in particular, 
\[
x(t_0)= K(t_0) \begin{bmatrix}
	u_0\\
	0
\end{bmatrix}+K(t_0) \begin{bmatrix}
	0\\
	v(t_0)
\end{bmatrix}=K(t_0) \begin{bmatrix}
	u_0\\
	v(t_0)
\end{bmatrix}, \quad u_0 \in \reals^d.
\]  
These considerations lead to the representations of the canonical subspaces
\begin{eqnarray*}
S_{can}(t)&:=& \im K (t)\begin{bmatrix}
		I_d\\
 0
\end{bmatrix} =\ \ker \begin{bmatrix}
	0 & I_a
\end{bmatrix}K^{-1} (t) \subset \reals^m, 
\\
N_{can}(t)&:=&\im K(t) \begin{bmatrix}
	0 \\I_{a}
\end{bmatrix}  =\ \ker \begin{bmatrix}
	I_{d} & 0
\end{bmatrix}K^{-1} (t)\subset \reals^m,
\end{eqnarray*}
with 
\[
S_{can}(t)\oplus N_{can}(t) = \reals^m, \quad t \in \mathcal I
\]
and the time-varying  canonical projector 
\[
\pPi_{can}(t):=K(t) \begin{bmatrix}
	I_d & 0 \\
	0 & 0
\end{bmatrix} K^{-1}(t)
\]
fulfilling
\[
\ker \pPi_{can}= N_{can}, \quad \im \pPi_{can}=S_{can}.
\]
Recall that this canonical projector has been introduced as completely-decoupling projector function in the context of the projector based decoupling, see e.g. \cite[Definition 2.37]{CRR}.

\subsection{About uniqueness}\label{sec:Uniqueness}

The known uniqueness of the canonical subspaces $S_{can}$ and $N_{can}$, and therefore also of the canonical projector $\pPi_{can}$, can be recognized using results for the SCF.
\medskip

For all linear DAEs, even if they are not regular but transformable into SCF, the SCF is unique in the sense that the dimensions of the ODE and the pure
 DAE are unique, and that both parts are unique up to equivalence transformations.\footnote{Let us clarify here some different use of \textit{regularity} in literature. DAEs that are transformable into SCF are called \textit{regular} in \cite{BergerIlchmann}, \textit{quasi-regular} in \cite{CRR} and \textit{almost regular} in \cite{commonground2024}. In this article,  the notion \textit{regularity} is used in the sense of \cite{commonground2024}. }

\begin{theorem}\cite[Theorem 2.1]{BergerIlchmann}\footnote{We slightly adapted this theorem from \cite{BergerIlchmann}, were only strictly lower triangular matrices were considered.}  [Uniqueness of SCF]\label{th:Uniqueness_SCF}
Let \( d, a, \tilde{d}, \tilde{a}  \in \mathbb{N}_0 \), $m=d+a=\tilde{d}+\tilde{a}$, and let
\[
\Omega \in C(\mathcal I, \mathbb{R}^{d \times d}), \quad \tilde{\Omega} \in C(\mathcal I, \mathbb{R}^{\tilde{d} \times \tilde{d}}),
\quad \mbox{and} \quad
N \in C(\mathcal I, \mathbb{R}^{a \times a}), \quad \tilde{N} \in C(\mathcal I, \mathbb{R}^{\tilde{a} \times \tilde{a}}),
\]
where both, \( N \) and \( \tilde{N} \), are pointwise strictly lower or strictly upper triangular matrix functions. If  there exist nonsingular 
\[
L \in C(\mathcal I,\mathbb{R}^{m \times m} ), \quad K \in C^1(\mathcal I, \mathbb{R}^{m \times m})
\]
such that
\[
\left\{ \begin{bmatrix}
	I_d & 0\\
	0 & N
\end{bmatrix},\begin{bmatrix}
	\Omega &0 \\
	0& I_a
\end{bmatrix} \right\} \quad \overset{L,K}{\sim} \quad \left\{ \begin{bmatrix}
	I_{\tilde{d}} & 0\\
	0 & \tilde{N}
\end{bmatrix},\begin{bmatrix}
	\tilde{\Omega} &0 \\
	0& I_{\tilde{a}}
\end{bmatrix}
\right\},
\]
then:
\begin{itemize}
    \item[(i)] \( d = \tilde{d} \), \( a = \tilde{a} \),
    \item[(ii)] \( L = \begin{bmatrix} L_{11} & 0 \\ 0 & L_{22} \end{bmatrix} \), 
                \( K = \begin{bmatrix} K_{11} & 0 \\ 0 & K_{22} \end{bmatrix} \),
                with \( L_{11} = K_{11}^{-1} , \quad  L_{22}=(K_{22}+NK_{22}')^{-1}\),
    \item[(iii)] 
    \(
		\left\{
    I_d, \Omega
		\right\}
		\overset{K_{11}^{-1}, K_{11}}{\sim}
    \left\{
    I_d, \tilde{\Omega}
		\right\}, \quad 
    \left\{
    N, I_a
		\right\}
		\overset{L_{22}, K_{22}}{\sim}
    \left\{
    \tilde{N},I_a
		\right\}.
    \)
\end{itemize}
\end{theorem}

\begin{proof}
A proof for strictly lower triangular $N$ was given in \cite{BergerIlchmann}. For strictly upper triangular $N$, it follows with permutation matrices that invert the order of rows/columns.
\end{proof}

The canonical subspaces, which are unique, may be represented in different ways, i.e., using different bases.
Therefore, for a DAE, there is neither a unique ODE in $\reals^d$ nor a unique  pure DAE.

\begin{corollary} \label{cor:uniqSCF}
For
\[
\{E, F\} \ \overset{L,K}{\sim} \ \left\{ \begin{bmatrix}
	I_d & 0\\
	0 & N
\end{bmatrix},\begin{bmatrix}
	\Omega & 0\\
	0 & I_{m-d}
\end{bmatrix} \right\} \
 \overset{\bar L, \bar K}{\sim} \
 \left\{ \begin{bmatrix}
	I_{\tilde{d}} & 0\\
	0 & \tilde{N}
\end{bmatrix},\begin{bmatrix}
	\tilde{\Omega} &0 \\
	0& I_{\tilde{n}}
\end{bmatrix}
\right\}
\]
with
\[
\bar L =\begin{bmatrix}
	\bar K_{11}^{-1} & 0 \\
	0 & \bar L_{22}
\end{bmatrix}, \quad \bar K =\begin{bmatrix}
	\bar K_{11} & 0 \\
	0 & \bar K_{22}
\end{bmatrix}
\] 
and
\[
\hat{K}:=K \begin{bmatrix}
	\bar K_{11} & 0 \\
	0 & \bar K_{22}
\end{bmatrix}, 
\]
it holds
\[
S_{can}(t)= \im K (t)\begin{bmatrix}
	I_d\\
	0
\end{bmatrix} =\im \hat{K} (t)\begin{bmatrix}
	I_d\\
	0
\end{bmatrix},
\quad
N_{can}(t)=\im K(t) \begin{bmatrix}
	0\\
	I_{m-d}
\end{bmatrix}=\im \hat K(t) \begin{bmatrix}
	0\\
	I_{m-d}
\end{bmatrix} ,
\]
\[
\pPi_{can}(t)=K(t) \begin{bmatrix}
	I_d & 0 \\
	0 & 0
\end{bmatrix} K^{-1}(t)= \hat K(t) \begin{bmatrix}
	I_d & 0 \\
	0 & 0
\end{bmatrix} \hat K^{-1}(t).
\]
Moreover, for
\[
\begin{bmatrix}
	u(t)\\
	v(t)
\end{bmatrix}:= K^{-1}(t) x(t), \quad \begin{bmatrix}
	\hat u(t)\\
	\hat v(t)
\end{bmatrix}:=\hat{K}^{-1}(t) x(t)=\begin{bmatrix}
	\bar K_{11}^{-1} & 0 \\
	0 & \bar K_{22}^{-1}
\end{bmatrix}\begin{bmatrix}
	u(t)\\
	v(t)
\end{bmatrix}
\]
and the two ODEs
\[
u'=\Omega u, \quad \hat u' =\hat \Omega \hat u,
\]
it holds
\begin{eqnarray}
\hat \Omega= \bar K_{11}^{-1} \Omega \bar K_{11}+\bar K_{11}^{-1} \bar K_{11}' .
\label{eq:Omega_hat}
\end{eqnarray}

\end{corollary}

\begin{proof}
The claim follows straight forward. 
\end{proof}

Not surprisingly, since we are confronted with time-varying transformations, they may change the stability behavior. 

\begin{example} \label{ex:ODEK11-1}
For the ODE
\[
u'=2u, 
\]
and the transformation with $K_{11}(t)=e^{-5t}$
we obtain
\[
\hat u'=\left(e^{5t}2e^{-5t}+e^{5t}(-5)e^{-5t}\right)\hat u=-3\hat u.
\]
\end{example}

\begin{example} (cf. Example \ref{ex:HMM98}) \label{ex:ODEK11-2}
For the ODE
\[
u'=\lambda u
\]
and the transformation with $K_{11}(t)=\sqrt{ \left( \eta t - 1\right)^{2} + 1 }=:\gamma(t)$ we obtain
\begin{small}
\begin{eqnarray*}
\hat u'&=& \left(\frac{1}{\gamma(t)} \lambda  \gamma(t) + \frac{1}{ \gamma(t)}\frac{2\eta^2 t -  2\eta }{2\gamma(t)}\right )\hat u
= \left(\lambda+ \frac{\eta^2 t -  \eta }{  \gamma(t)^2 }\right)\hat u.
\end{eqnarray*}
\end{small}
\end{example}

With these examples we underline that, as soon as time-varying transformations are involved, the stability behavior might change arbitrarily.
This motivates the following definition.

\begin{definition}
If a DAE is equivalently transformable into SCF, then we  call every explicit ODE part \eqref{eq:pureODE}  living in the configuration space $\reals^d$ for $d>0$ and resulting from such a transformation a \emph{pure} ODE.
\end{definition}

Observe that pure ODEs do not involve any derivatives of the right-hand side $q$ and  are not unique.

\begin{remark}
Above we agreed upon the name pure ODE part for \eqref{eq:pureODE} and do not use the notation essential underlying ODE (EUODE) for good reasons. The  label EUODE itself  was first used in \cite{AscherPetzold91} for special DAEs without any reference to an SCF and generalized in \cite{LinhMaerz} for arbitrary regular DAEs via a very special transformation into a structured SCF. EUODEs are, so to speak, compressed versions of the so-called projector-based inherent ODEs in the context of the tractability framework.  By definition, EUODEs  are also  pure ODE parts, however,  they are strictly bound to  specific transformations $L$ and $K$, which in turn arise from complete projector-based decouplings. An EUODE  shares with any pure ODE part the dimension $d$ and the property that it does not involve any derivatives of the right-hand side $q$. As it is verified in \cite{LinhMaerz},  to each regular DAE, one can form the compression in such a way that  the resulting  EUODE reflects the Lyapunov spectrum of the DAE. Therefore, among the pure ODEs parts of a DAE there is a spectrum preserving one. However, it is still an open question how a group of transformations could be figured out in order to obtain a SCF that contains an spectrum preserving pure ODE part without recourse to the projector-based framework. 
\end{remark}

\subsection{Block-structured canonical forms of regular DAEs}

The computation of an SCF for time-varying DAEs is a nontrivial transformation in general. Indeed, there are further canonical forms which play their role within the projector-based \cite{CRR} and the strangeness-reduction \cite{KuMe2024} frameworks, both of them designed in a constructive way.
  We use them here also to motivate the structural assumptions in Section \ref{sec:preliminary} and will show in Section \ref{sec:T-S-canonical} that they (almost) fulfill them.

Following the notation of \cite[Section 2.10.1]{CRR} a  time-varying pair
\[
\left\{ 
\begin{bmatrix}
	I_{d} & E_{12}\\
	0 & N
\end{bmatrix}, \quad \begin{bmatrix}
	F_{11} & 0 \\
	F_{21} & I_{a}
\end{bmatrix}
\right\},
\]
for a strictly block upper  triangular matrix function
\[
N = \begin{bmatrix}
    0 & N_{12} & \cdots & N_{1\mu} \\
    & \ddots & \ddots & \vdots \\
    & & \ddots & N_{\mu-1,\mu} \\
    & & & 0
    \end{bmatrix}\begin{matrix}
		\ell_1 \\
		\vdots\\
		\ell_{\mu-1}\\
		\ell_{\mu}
		\end{matrix}, \quad \quad \ell_i \in \N, \quad \ell = \sum_{i=1}^{\mu} \ell_i=a,
\]
is said to be in 
\begin{enumerate}
    \item \textbf{T-canonical form} if
    \[
    E_{12} = 0, 
		\]
		 and each \( N_{i,i+1} \) has full \emph{column} rank for \( i = 1, \ldots, \mu - 1 \), i.e., \( N =N^c \in \SUT_{column}\) using the notation from Appendix \ref{Appendix_block}.
		
Note that this implies decreasing $\ell_1 \geq \cdots \geq \ell_{\mu} $ and $\rank N_{i,i+1} = \ell_{i+1} $.		
\medskip		

For DAEs, a  T-canonical form primarily shows a pure ODE part:
\begin{eqnarray*}
x_1'+F_{11} x_1&=&q_1\\
N^c x_2'+x_2 &=& q_2-F_{21}x_1,
\end{eqnarray*}
whereas in case of a complete decoupling $F_{21}=0$ is given and therefore the T-canonical form is already in SCF \cite{CRR}. 
		\item \textbf{S-canonical form} if
    \[
    F_{21} = 0, \quad E_{12} = \begin{bmatrix} 0 & (E_{12})_2 & \cdots &(E_{12})_{\mu} \end{bmatrix}, \quad \]
     and each \( N_{i,i+1} \) has full \emph{row} rank for \( i = 1, \ldots, \mu - 1 \), i.e., \( N =N^r\in \SUT_{row}\) using the notation from Appendix \ref{Appendix_block}.

    Note that this implies increasing $\ell_1 \leq \cdots \leq \ell_{\mu} $ and $\rank N_{i,i+1} = \ell_{i} $.
\medskip

For DAEs, a S-canonical form primarily shows a pure DAE:
\begin{eqnarray*}
x_1'+E_{12}x_2'+F_{11}x_1&=&q_1,\\
N^r x_2'+x_2 &=& q_2.
\end{eqnarray*}

\end{enumerate}

Having the T- or S-canonical form with block sizes $d, \ell_1, \ldots, \ell_{\mu}$, we know immediately the canonical characteristics $r, \theta_0, \ldots, \theta_{\mu}$, see Appendix \ref{Appendix_block}.
\medskip

We will show how each of these canonical forms can be transformed into a preliminary stage of an SCF according to the next Section \ref{sec:preliminary}. In this context, the properties of block structured matrix valued functions with decreasing or increasing block sizes, that are summarized in Appendix \ref{Appendix_block}, play their role. 
Note that in both cases
$
a=\sum_{i=1}^{\mu} \ell_i=m-d$. 

\subsection{Pairs in QuasiSCF} \label{sec:QuasiSCF}
 In \cite{commonground2024},  a couple of different definitions  of linear regular DAEs were shown to be equivalent. Here, we
 focus only on those that involve block structured standard canonical forms and its variants from \cite{SSCF25}. 
We deal with the block structures from last section, which are described in detail in the Appendix \ref{Appendix_block}. For simplicity, if either strictly upper block matrices $\SUT_{columns}$ or $\SUT_{row}$ can be considered, we use the abbreviation $\SUT_{c/r}$.

\begin{proposition}  (Special case of \cite[Theorem 8.1]{commonground2024}) \label{def:regEF_N}
A pair of matrix functions $\{E,F\}$ is regular, iff 
\[
\{E,F\} \ \sim \ \left\{ \begin{bmatrix}
	I_{d} & 0 \\
	0 & N^{c/r}
\end{bmatrix},\begin{bmatrix}
	\Omega & 0 \\
	0 & I_{a}
\end{bmatrix} \right\},
\]
with 
\[
N^{c/r} \in \SUT_{c/r}, \quad  \Omega :\mathcal I\rightarrow \reals^{d \times d}.
\]  
\end{proposition}
The starting point for the following considerations is based on the existence of a block structured SSCF with decreasing or increasing block sizes of the constant nilpotent matrix. In particular, we can assume that the constant matrix $N^{(C)}$ has an elementary form  $N^{(E_{c/r})}$, see \cite{SSCF25}:
\begin{itemize}
	\item  for $\ell_1 \geq \cdots \geq \ell_{\mu} $: $N^{(E_c)}$ (\textbf{E}lementary form with full \textbf{c}olumn rank) defined by 
	\[
	N^{(E_c)}_{i,i+1}=\begin{bmatrix}
		I_{\ell_{i+1}} \\
		0
	\end{bmatrix}  , \quad N^{(E_c)}_{i,j}=0 \quad \mbox{ for}  \ j\neq i+1,
	\]
	\item  for $\ell_1 \leq \cdots \leq \ell_{\mu} $: $N^{(E_r)}$ (\textbf{E}lementary form with full \textbf{r}ow rank) defined by 
	\[
	N^{(E_r)}_{i,i+1}=\begin{bmatrix}
		I_{\ell_{i}} &
		0
	\end{bmatrix}  , \quad  N^{(E_r)}_{i,j}=0 \quad \mbox{for}  \ j\neq i+1.
	\]
\end{itemize}
Both elementary matrices are illustrated in Figures \ref{fig:Bilder_col} and \ref{fig:Bilder_row} and permit the following characterization of regularity, where $N^{(E_{c/r})}$ means either $N^{(E_{c})}$ or $N^{(E_{r})}$.

\begin{proposition}
 A pair of matrix functions $\{E,F\}$ is regular, iff  
\[
\{E,F\} \ \sim \  \left\{ \begin{bmatrix}
	I_{d} & 0 \\
	0 & N^{(E_{c/r})}
\end{bmatrix}, \begin{bmatrix}
	\Omega & 0 \\
	0 & I_{a}
\end{bmatrix} \right\},
\]
with  $\Omega :\mathcal I\rightarrow \reals^{d \times d}$.
\end{proposition}
The equivalence of these two possible definitions of regularity has been shown in \cite{SSCF25}. With regard to our purposes here,  we show yet another possibility.

\begin{proposition} \label{prof:quasiSCF}
A pair of matrix functions $\{E,F\}$ is regular, iff 
\[
\{E,F\} \ \sim \ \left\{ \begin{bmatrix}
	I_{d} & 0 \\
	0 & N^{(E_{c/r})}
\end{bmatrix}, \begin{bmatrix}
	\Omega & 0 \\
	0 & R
\end{bmatrix} \right\},
\]
with
\begin{itemize}
	\item $\Omega :\mathcal I\rightarrow \reals^{d \times d}$,
	\item  $R :\mathcal I\rightarrow \reals^{a \times a}$, whereas $R$ is block-structured upper triangular with either ascending or descending block sizes corresponding to $N^{(E_{c/r})}$ and nonsingular, i.e., using the notation from Appendix \ref{Appendix_block}, $R \in \BUT_{nonsingular}$.
\end{itemize}

\end{proposition}
\begin{proof}
One direction of the equivalence to Definition \ref{def:regEF_N} follows with
\[
L=\begin{bmatrix}
	I_d & 0 \\
	0 & R^{-1}
\end{bmatrix}, \quad  K=I_m,
\]
 and Lemma \ref{lem:R_BUT}, the other with Corollary \ref{cor:Rcr} and either
\[
L=\begin{bmatrix}
	I_d & 0 \\
	0 & (R^c)^{-1}
\end{bmatrix}, \quad K=I_m,
\]
or
\[
L=\begin{bmatrix}
	I_d & 0 \\
	0 & (I_{a}+N^{(E_{r})}((R^r)^{-1})')^{-1}
\end{bmatrix}, \quad K=\begin{bmatrix}
	I_d & 0 \\
	0 & (R^{r})^{-1}
\end{bmatrix}.
\]

\end{proof}

This motivates the following definition
\begin{definition} \label{def:QuasiSCF}
A  pair 
\[
\left\{  \begin{bmatrix}
	I_d & 0 \\
	0 & N^{(E_{c/r})}
\end{bmatrix} , \begin{bmatrix}
\Omega & 0 \\
0 & R
\end{bmatrix} \right\}
\]
is in QuasiSCF, iff $R \in \BUT_{nonsingular}$, with block sizes corresponding to $N^{(E_{c/r})}$.
\end{definition}

From the proof of Proposition \ref{prof:quasiSCF} it follows that for the   $K$ used to transform an SCF into a QuasiSCF (or vice versa) it holds 
\[
\im K \begin{bmatrix}
	I_d\\
	0
\end{bmatrix} = \im \begin{bmatrix}
	I_d\\
	0
\end{bmatrix}, \quad \im K \begin{bmatrix}
	0\\
	I_a
\end{bmatrix} = \im \begin{bmatrix}
	0\\
	I_a
\end{bmatrix},
\]
such that $S_{can}$, $N_{can}$, $\pPi_{can}$, a pure ODE and a pure DAE can be characterized using the  QuasiSCF instead of an SCF, as we will see later in more detail.

\section{DAEs in PreSCF: A preliminary stage of a QuasiSCF}\label{sec:preliminary}

Our computation of an SCF or even an SSCF we first realize that a transformation into a form
\[
\left\{  \begin{bmatrix}
	I_d & 0 \\
	0 & N^{(E_{c/r})}
\end{bmatrix} , \begin{bmatrix}
F_{11} & F_{12} \\
F_{21} & F_{22}
\end{bmatrix} \right\}
\]
is always possible for regular DAEs, since indeed a transformation into an SSCF exists. However, since an SSCF is difficult to find, we think of transformations that leave the first matrix unchanged, see Appendix \ref{Appendix:Equivalence}, Lemma \ref{lem:M12} and Lemma \ref{lem:M21}. For nonsingular $F_{22}$, they provide the following results that are decisive for the next steps.

\begin{lemma}\label{lem:M12F22}
For every regular  pair in the form
\[
\{E_0,F_0\}=\left\{ \begin{bmatrix}
	I & 0 \\
	0 & N^{(E_{c/r})}
\end{bmatrix}, \begin{bmatrix}
	F_{11}& F_{12} \\
	F_{21} & F_{22}
\end{bmatrix} \right\}, 
\]
with nonsingular $F_{22}$ and the transformation matrix functions
\[
L=\begin{bmatrix}
	I_{d} & -F_{12} F_{22}^{-1} \\
	0 & I_{a}
\end{bmatrix}, \quad 
K=\begin{bmatrix}
	I_{d} & F_{12} F_{22}^{-1}N^{(E_{c/r})} \\
	0 & I_{a}
\end{bmatrix}
\]
it holds
\begin{eqnarray*}
LE_0K&=&E_0, \\ 
LF_0K+LE_0K'&=&\begin{bmatrix}
F_{11} - F_{12} F_{22}^{-1}F_{21}  &    \left(F_{11}-F_{12} F_{22}^{-1}F_{21}+(-F_{12} F_{22}^{-1})' \right) F_{12} F_{22}^{-1} N^{(E_{c/r})}\\F_{21} &   F_{22} + F_{21} F_{12} F_{22}^{-1}N^{(E_{c/r})} 
\end{bmatrix}.
\end{eqnarray*}

\end{lemma}
\begin{proof}
Straight forward computation gives the result.
\end{proof}

\begin{lemma}\label{lem:M21F22}
For every regular  pair in the form
\[
\{E_0,F_0\}=\left\{ \begin{bmatrix}
	I & 0 \\
	0 & N^{(E_{c/r})}
\end{bmatrix}, \begin{bmatrix}
	F_{11}& F_{12} \\
	F_{21} & F_{22}
\end{bmatrix} \right\}, 
\]
with nonsingular $F_{22}$ and the  transformation matrix functions
\[
L=\begin{bmatrix}
	I_{d} & 0 \\
	N^{(E_{c/r})}F_{22}^{-1}F_{21} & I_{a}
\end{bmatrix}, \quad 
K=\begin{bmatrix}
	I_{d} & 0 \\
	-F_{22}^{-1}F_{21} & I_{a}
\end{bmatrix}
\]
it holds
\begin{eqnarray*}
LE_0K&=&E_0, \\
LF_0K+LE_0K'&=&\begin{bmatrix}
	F_{11} - F_{12} F_{22}^{-1}F_{21}& F_{12} \\
	N^{(E_{c/r})}F_{22}^{-1}F_{21} \left(F_{11}    -  F_{12}F_{22}^{-1}F_{21} -N^{(E_{c/r})}(F_{22}^{-1}F_{21})'\right) & F_{22}+N^{(E_{c/r})}F_{22}^{-1}F_{21}F_{12}
\end{bmatrix}.
\end{eqnarray*}
\end{lemma}

\begin{proof}
Straight forward computation gives the result.
\end{proof}

It is easy to recognize now that $N^{(E_{c/r})}$ is introduced in expressions of the second matrix function, such that repeating such transformations  makes it possible to take advantage of the nilpotency. 
Indeed, these two lemmas permit an iterative computation of a QuasiSCF if in all steps $F_{22} \in \BUT_{nonsingular}$ is given and the nilpotency of the matrix $N^{(E_{c/r})}$ guarantees a finite number of iterations, see Section \ref{sec:Procedure}.
 To ensure the feasibility, we again use the block structures introduced above, see Figures \ref{fig:Bilder_col} and \ref{fig:Bilder_row}. Some properties of all these matrix functions and in particular the set of nonsingular matrix functions $\BUT_{nonsingular} \subset \BUT$  are summarized in Appendix \ref{Appendix_block}.

\begin{figure}
\includegraphics[width=\textwidth]{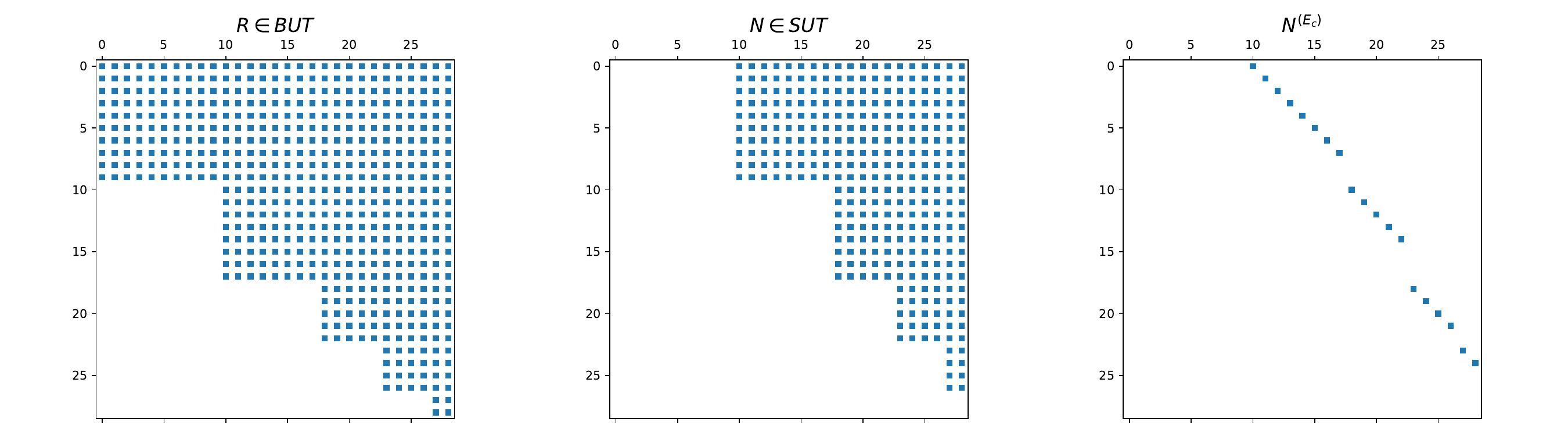}
\caption{Visualization of the block structure for decreasing block sizes $\ell_1=10, \ell_2=8, \ell_3=5, \ell_4=4, \ell_5=2$.}
\label{fig:Bilder_col}
\end{figure} 

\begin{figure}
\includegraphics[width=\textwidth]{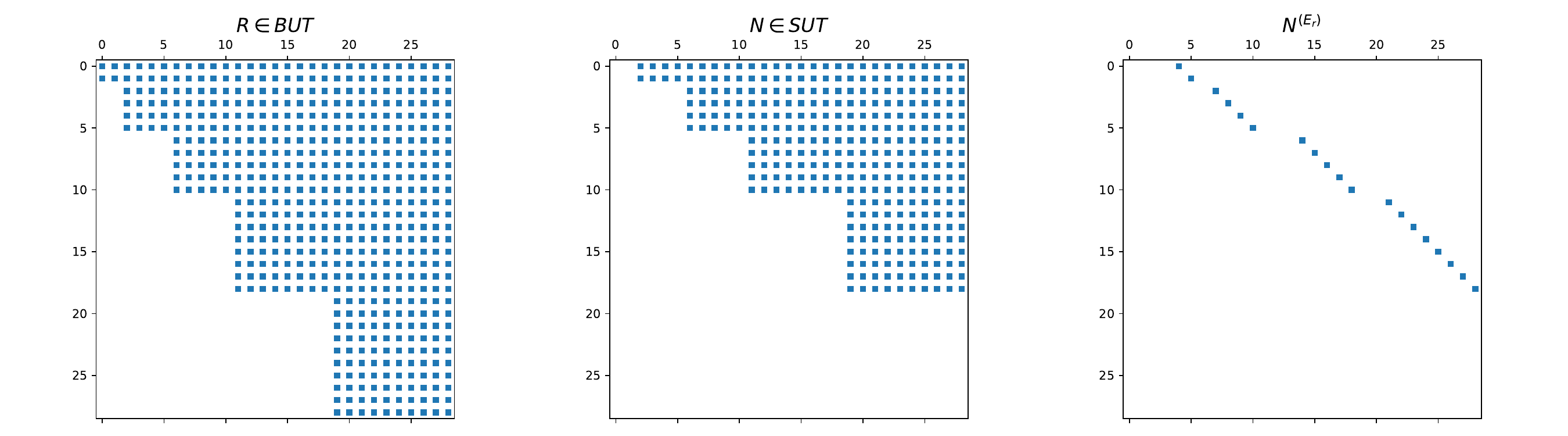}
\caption{Visualization of the block structure for increasing block sizes $\ell_1=2, \ell_2=4, \ell_3=5, \ell_4=8, \ell_5=10$.}
\label{fig:Bilder_row}
\end{figure} 

\begin{definition} \label{def:prel_SCF}
A  pair 
\[
\left\{  \begin{bmatrix}
	I_d & 0 \\
	0 & N^{(E_{c/r})}
\end{bmatrix} , \begin{bmatrix}
F_{11} & F_{12} \\
F_{21} & F_{22}
\end{bmatrix} \right\}
\]
is in a preliminary stage of an SCF (PreSCF) iff 
$F_{22}$ and $F_{21}F_{12}$ are block upper triangular matrices with the same block sizes as $N^{(E_{c/r})}$ and $F_{22}$ is nonsingular.\\

In short, this means  $F_{22} \in \BUT_{nonsingular}$,  $F_{21}F_{12}\in \BUT$  in terms of the notation from Appendix \ref{Appendix_block}.
\end{definition}
For DAEs in a PreSCF the algorithm below will produce an SCF after a finite number of steps due to Lemmas \ref{lem:M12F22} and \ref{lem:M21F22}, since all involved pairs of matrix functions have nonsingular $F_{22}$ thanks to the structural properties.

\section{A subsequent  procedure to compute  an SSCF}\label{sec:Procedure}

To obtain an SSCF we consider several transformation steps with $E_0=E_1=E_2=E_4$:
\[
\left\{E, F \right\} \ 
\underset{\overset{\longrightarrow}{\mbox{\small{Step 0}}}}{\overset{L_0, K_0}{\sim}} \ \underbrace{\left\{E_0, F_0 \right\}}_{PreSCF}\ 
\underset{\overset{\longrightarrow}{\mbox{\small{Step 1}}}}{\overset{L_1, K_1}{\sim}} \ 
\underbrace{\left\{E_1, F_1 \right\}}_{(F_1)_{12}=0} \ 
\underset{\overset{\longrightarrow}{\mbox{\small{Step 2}}}}{\overset{L_2, K_2}{\sim}} \ 
\underbrace{\left\{E_2, F_2 \right\}}_{QuasiSCF} \ 
\underset{\overset{\longrightarrow}{\mbox{\small{Step 3}}}}{\overset{L_3, K_3}{\sim}} \ 
\underbrace{\left\{E_3, F_3 \right\}}_{SCF} \ 
\underset{\overset{\longrightarrow}{\mbox{\small{Step 4}}}}{\overset{L_4, K_4}{\sim}} \ 
\underbrace{\left\{E_4, F_4 \right\}}_{SSCF}.
\]
The transformation matrices $L_0, K_0$ depend on the structure of the given DAE. Here, we discuss them for some particular cases.
\medskip
 
Assuming  that a  pair of matrix functions is in a PreSCF, the transformation matrices $L_i, K_i$, $i=1,2,3,4$ are the ones used in the $i$-th Step according to the notation below, and  result systematically from further intermediate steps.

For Steps 1 and 2, we use the two elementary equivalence transformations that preserve the structure of the first constant matrix from Lemmas \ref{lem:M12F22} and \ref{lem:M21F22} and  take advantage of the nilpotency in iterations.

\subsection{Step 1: Obtaining a pure ODE part}

In this step we construct the equivalence transformation
\[
\underbrace{\left\{E_0, F_0 \right\}}_{PreSCF}\ \overset{L_1, K_1}{\sim} \ \left\{E_1, F_1 \right\}, \quad \mbox{with} \quad (F_1)_{12}=0.
\]

To this end, we use Lemma \ref{lem:M12F22} to construct a sequence $F^{(k)}$, $k\in \N$, with $F_{12}^{(k)} \rightarrow 0 $, $F^{(0)}=F_0$. In a first iteration we obtain
\[
 F_{12}^{(1)}=((F_{11}-F_{12}F_{22}^{-1}F_{21})F_{12}F_{22}^{-1}+(F_{12}F_{22}^{-1})')N^{(E_{c/r})},\quad F_{22}^{(1)}=F_{22}+F_{21}F_{12}F_{22}^{-1}N^{(E_{c/r})}.
\]
We observe two aspects:
\begin{itemize}
	\item Since $F_{21}F_{12}F_{22}^{-1}N^{(E_{c/r})}$ is strictly upper block triangular, the diagonal blocks of  $F_{22}$ are preserved in $F_{22}^{(1)}$, such that it is nonsingular.
	\item $F_{12}^{(1)}$ contains at the end the product with a nilpotent matrix. Therefore, repeating this procedure $F_{12}^{(\mu)}=0$ and $R:=F_{22}^{(\mu)}$, $\Omega:=F_{11}^{(\mu)}$.
\end{itemize}
Therefore,
\[
\left\{E_1, F_1 \right\} =\left\{ \begin{bmatrix}
		I_d & 0\\
		0 & N^{(E_{r/c})}
	\end{bmatrix} , \begin{bmatrix}
		\Omega & 0\\
		F_{21} & R
	\end{bmatrix}\right\}.
\]
We see that with this step we
 already obtained a pure ODE.

\subsection{Step 2: Obtaining a pure DAE part}

In this step we construct the equivalence transformation
\[
\left\{E_1, F_1 \right\} \ \overset{L_2, K_2}{\sim} \ \underbrace{\left\{E_2, F_2 \right\}}_{QuasiSCF}.
\]

We use Lemma \ref{lem:M21F22} to construct a sequence $F^{(k)}$, $k \in \N$ with $F_{21}^{(k)} \rightarrow 0 $, $F^{(0)}=F_1$. Let us first note that for $F^{(0)}=F_0$ a first iteration would lead to
\[
F_{21}^{(1)}=N^{(E_{c/r})}(F_{22}^{-1}F_{21}(F_{11}-F_{12}F_{22}^{-1}F_{21})-(F_{22}^{-1}F_{21})'),\quad F_{22}^{(1)}=F_{22}+N^{(E_{c/r})} F_{22}^{-1}F_{21}F_{12}.
\]
Now, we observe that
\begin{itemize}
  \item Since $N^{(E_{c/r})} F_{22}^{-1}F_{21}F_{12}$ is strictly upper block triangular, the diagonal blocks of  $F_{22}$ are preserved in $F_{22}^{(1)}$, such that it is nonsingular.
	\item $F_{21}^{(1)}$ starts with the product with a nilpotent matrix. Therefore, repeating this procedure $F_{21}^{(\mu)}=0$.
	\item If Step 2 is done after Step 1, i.e., $F^{(0)}=F_1$, then $F_{12}^{(j)}=0$,  $F_{22}^{(j)}=R$, $F_{11}^{(j)}=\Omega$ remain unchanged for $j=0,1, \ldots, \mu$.
\end{itemize}
 Note that the order of Step 1 and Step 2 could also have been inverted. We prefer this order, since then we first obtain a pure ODE with the final $\Omega$.

\subsection{Steps 3 and 4: Computation of an SCF and an SSCF}
As soon as we have a representation of the form 
\[
E_2= \begin{bmatrix}
I_d & 0 \\
0 & N^{(E_{c/r})}
\end{bmatrix}, \quad F_2=\begin{bmatrix}
\Omega & 0 \\
0 & R
\end{bmatrix}\,
\]
for a nonsingular block upper triangular $R$, with 
\[
L_3 = \begin{bmatrix}
I_d & 0 \\
0 & R^{-1}
\end{bmatrix}, \quad K_3=I_m
\]
we obtain an SCF. In a last step, if desired, an SSCF can be computed  with particular $L_{SSCF}$ and $K_{SSCF}$ for the pure DAE as described in \cite{SSCF25} and
\[
L_4 = \begin{bmatrix}
I_d & 0 \\
0 & L_{SSCF}
\end{bmatrix}, \quad K_4 = \begin{bmatrix}
I_d & 0 \\
0 & K_{SSCF}
\end{bmatrix}.
\]

Since $K_3$ and $K_4$ transform algebraic variables only, they are not crucial for the canonical spaces and therefore we will not go into them any further, cf. Section \ref{sec:Pican}.

\subsection{A detailed example}

We consider the following example to illustrate on the one hand all the described steps and on the other hand the importance of the regularity requirement.
\begin{example} \label{ex:BergIl}
 \cite[Example 5.6.]{BergerIlchmann} 
\[
E=
\left[\begin{matrix}\sin{\left(t \right)} & \cos{\left(t \right)} & 0\\0 & 0 & 0\\- \frac{\sin{\left(2 t \right)}}{2} & \sin^{2}{\left(t \right)} & 0\end{matrix}\right]
,\quad 
F=
\left[\begin{matrix}- \sin{\left(t \right)} + \cos{\left(t \right)} & - \sin{\left(t \right)} - \cos{\left(t \right)} & 0\\\cos{\left(t \right)} & - \sin{\left(t \right)} & 0\\\sin^{2}{\left(t \right)} & \frac{\sin{\left(2 t \right)}}{2} & - t^{2} - 1\end{matrix}\right].
\]
In \cite{BergerIlchmann} this example is used to illustrate their transformation into SCF that does not presuppose regularity,  but real analytic $E,F$. Indeed, the nilpotent matrix of the SCF found there has rank drops at zeros of $\sin$, just like the matrix $E$.
 Therefore it is obvious that on $\reals$ no equivalent transformation into an SSCF  is possible. 
\medskip

Here, for regularity we require 
$
\mathcal I \subset \reals \setminus \{ n\pi \mid n \in \mathbb{Z} \} 
$,
such that $m=3, \mu=2, r=2, \theta_0=1, d=1, a=2$ on $\mathcal I$, $N^{(E_c)}=N^{(E_r)}=\begin{bmatrix}
	0 & 1\\ 0 &0
\end{bmatrix}$.
\medskip

Step 0:  With

\[
L_0=
\left[\begin{matrix}1 & -1 - \frac{\cos{\left(t \right)}}{\sin{\left(t \right)}} & 0\\\cos{\left(t \right)} & 0 & 1\\0 & 1 & 0\end{matrix}\right]
,\quad 
K_0=
\left[\begin{matrix}\frac{1}{\sin{\left(t \right)}} & 0 & - \frac{\cos{\left(t \right)}}{\sin{\left(t \right)}}\\0 & 0 & 1\\0 & 1 & 0\end{matrix}\right]
,\quad 
\]
we obtain the required form of a PreSCF:
\[
E_0=
\left[\begin{array}{c|cc}1 & 0 & 0\\
\hline
0 & 0 & 1\\0 & 0 & 0\end{array}\right]
,\quad 
F_0=
\left[\begin{array}{c|cc}\frac{\sin{\left(2 t \right)} + 2}{\cos{\left(2 t \right)} - 1} & 0 & \frac{\sqrt{2} \sin{\left(t + \frac{\pi}{4} \right)}}{\sin^{2}{\left(t \right)}}\\
\hline
- \cos{\left(t \right)} + \frac{1}{\sin{\left(t \right)}} & - t^{2} - 1 & - \frac{\cos{\left(t \right)}  }{\sin{\left(t \right)}}\\\frac{\cos{\left(t \right)}  }{\sin{\left(t \right)}} & 0 & - \frac{1}{\sin{\left(t \right)}}\end{array}\right]
.
\]

Step 1:
\begin{enumerate}

\item First (and in this case last) iteration step:
\[
L_{1}^{(1)}=
\left[\begin{array}{c|cc}1 & 0 & 1 + \frac{\cos{\left(t \right)}}{\sin{\left(t \right)}}\\
\hline
0 & 1 & 0\\0 & 0 & 1\end{array}\right]
,\quad 
K_{1}^{(1)}=
\left[\begin{array}{c|cc}1 & 0 & 0\\
\hline 
0 & 1 & 0\\0 & 0 & 1\end{array}\right]
,\quad 
\]
$E_1:=E_1^{(1)}=E_0$, $L_1:=L_{1}^{(1)}$, $K_1:=K_{1}^{(1)}$,
\[
F_1:=F_{1}^{(1)}=
\left[\begin{array}{c|cc}
-1 & 0 & 0 \\
\hline
- \cos{\left(t \right)} + \frac{1}{\sin{\left(t \right)}} & - t^{2} - 1 & - \frac{\cos{\left(t \right)}  }{\sin{\left(t \right)}}
\\
\frac{\cos{\left(t \right)}  }{\sin{\left(t \right)}} & 0 & - \frac{1}{\sin{\left(t \right)}}
\end{array}\right]
.
\]
\end{enumerate}
Step 2:
\begin{enumerate}
	\item First iteration step:
\[
L_{2}^{(1)}=
\left[\begin{array}{c|cc}1 & 0 & 0\\
\hline
- \cos{\left(t \right)} & 1 & 0\\0 & 0 & 1\end{array}\right]
,\quad 
K_{2}^{(1)}=
\left[\begin{array}{c|cc}1 & 0 & 0\\
\hline
- \frac{\sqrt{2} \cos{\left(t + \frac{\pi}{4} \right)}}{t^{2} + 1} & 1 & 0\\\cos{\left(t \right)} & 0 & 1\end{array}\right]
,\quad 
\]
$E_{2}^{(1)}=E_1=E_0$ and
\[
F_{2}^{(1)}=
\left[\begin{array}{c|cc}-1 & 0 & 0\\
\hline
\sqrt{2} \cos{\left(t + \frac{\pi}{4} \right)} & - t^{2} - 1 & - \frac{\cos{\left(t \right)}  }{\sin{\left(t \right)}}
\\
0 & 0 & - \frac{1}{\sin{\left(t \right)}}\end{array}\right]
.
\]
\item Second (and last) iteration step:
\[
L_{2}^{(2)}=
\left[\begin{array}{c|cc}1 & 0 & 0\\
\hline
0 & 1 & 0\\0 & 0 & 1\end{array}\right]
,\quad 
K_{2}^{(2)}=
\left[\begin{array}{c|cc}1 & 0 & 0\\
\hline
\frac{\sqrt{2} \cos{\left(t + \frac{\pi}{4} \right)}}{t^{2} + 1} & 1 & 0\\0 & 0 & 1\end{array}\right]
,\quad 
\]
$E_2:=E_2^{(2)}=E_{2}^{(1)}=E_1=E_0$, $L_2:=L_{2}^{(2)}L_{2}^{(1)}$, $K_2:=K_{2}^{(1)}K_{2}^{(2)}$, results in a QuasiSCF with
\[
F_2:=F_{2}^{(2)}=
\left[\begin{array}{c|cc}-1 & 0 & 0\\
\hline
0 & - t^{2} - 1 & - \frac{\cos{\left(t \right)}  }{\sin{\left(t \right)}}
\\
0 & 0 & - \frac{1}{\sin{\left(t \right)}}\end{array}\right]
.
\]
\end{enumerate}
Step 3:
\[
L_{3}=
\left[\begin{array}{c|cc}1 & 0 & 0\\
\hline
0 & - \frac{1}{t^{2} + 1} & \frac{\cos{\left(t \right)}}{t^{2} + 1}\\0 & 0 & - \sin{\left(t \right)}\end{array}\right], \quad K_3=I_3,
\]
leads to the SCF
\[
E_3 = E_0=\left[\begin{array}{c|cc}1 & 0 & 0\\
\hline
0 & 0 & - \sin{\left(t \right)}\\0 & 0 & 0\end{array}\right] , \quad F_3=
\left[\begin{array}{c|cc}-1 & 0 & 0\\
\hline
0 & 1 & 0\\0 & 0 & 1\end{array}\right].
\]

Step 4:
Finally, with
\[
L_4=\left[\begin{array}{c|cc}1 & 0 & 0\\
\hline
0 & - \frac{1}{\sin{\left(t \right)}} & 0\\0 & 0 & 1\end{array}\right] , \quad K_4=
\left[\begin{array}{c|cc}1 & 0 & 0\\
\hline
0 & - \sin{\left(t \right)} & 0\\0 & 0 & 1\end{array}\right],
\]
we obtain the SSCF
\[
E_4 = E_0=\left[\begin{array}{c|cc}1 & 0 & 0\\
\hline
0 & 0 & 1\\0 & 0 & 0\end{array}\right] , \quad F_4=
\left[\begin{array}{c|cc}-1 & 0 & 0\\
\hline
0 & 1 & 0\\0 & 0 & 1\end{array}\right].
\]
\medskip
Note that for the homogenous DAE the pure ODE $u'-u=0$, 
also given in \cite{BergerIlchmann}, is found already after Step 1.  However, it is not unique and different $L_0, K_0$ may provide different pure ODEs. We illustrate and discus this in Example \ref{ex:HMM98}.

\end{example}
In the following examples, for shortness, we will not specify the intermediate iteration steps of Step 1 and Step 2 anymore.


\section{Consequences for the canonical projector and pure ODEs} \label{sec:Pican}
If $K:=K_0 K_1 K_2 K_3 K_4$ is the transformation of the unknowns that was used above to transform an arbitrary DAE into an SSCF, then the canonical projector
\[
\pPi_{can}: =K\begin{bmatrix}
I_d & 0 \\
0 &0
\end{bmatrix} K^{-1} 
\]
can be computed as
\[
\pPi_{can}: =K_0 K_1 K_2\begin{bmatrix}
I_d & 0 \\
0 &0
\end{bmatrix} K_2^{-1} K_1^{-1}K_0^{-1},  
\]
since $K_3$ and $K_4$ transform algebraic variables only. Therefore, these two last steps will not be that important for latter considerations, such that we will focus on Steps 0, 1 and 2.

\begin{proposition}\label{prop:Pican}
For  every sufficiently smooth pair of matrix functions in a PreSCF,
\[
\left\{ \begin{bmatrix}
		I_d & 0\\
		0 & N^{(E_{r/c})}
	\end{bmatrix} , \begin{bmatrix}
		F_{11} & F_{12}\\
		F_{21} & F_{22}
	\end{bmatrix}\right\},
\]
the canonical spaces are
\[
S_{can}=\im \begin{bmatrix}
	I_d + AB  \\
	B 
\end{bmatrix}, \quad N_{can}=\im \begin{bmatrix}
	A  \\
	I_{a}
\end{bmatrix},
\]
and the canonical projector
$
\pPi_{can} $
is
\[
\pPi_{can} = \begin{bmatrix}
	I_d + AB & -A-ABA \\
	B & -BA
\end{bmatrix}
\]
for $A:=\sum_{i=1}^{\mu} A_i : \mathcal I\rightarrow \reals^{d \times a}$, $B:=\sum_{j=1}^{\mu} B_j : \mathcal I\rightarrow \reals^{ a\times d}$, resulting from
\begin{eqnarray*}
F^{(0)}&=& \begin{bmatrix}
		F_{11} & F_{12}\\
		F_{21} & F_{22}
	\end{bmatrix}, \\
	F^{(i)}&=&  \begin{bmatrix}
	I_d & -F_{12}^{(i-1)} (F_{22}^{(i-1)})^{-1} \\
	0 & I_{a}
\end{bmatrix} F^{(i-1)} \begin{bmatrix}
	I_d & A_{i} \\
	0 & I_{a}
\end{bmatrix}\\
&&+ \begin{bmatrix}
	0 & A_{i}' \\
	0 & 0
\end{bmatrix} , \quad i=1, \ldots, \mu \\
F^{(\mu+j)}&=&  \begin{bmatrix}
	I_d & 0 \\
	-(F_{22}^{(\mu+j-1)})^{-1} F_{21}^{(\mu+j-1)}N^{(E_{c/r})} & I_{a}
\end{bmatrix} F^{(\mu+j-1)} \begin{bmatrix}
	I_d & 0 \\
	B_j & I_{d}
\end{bmatrix}\\
&&+ 
	\begin{bmatrix}
	0 & 0 \\
	N^{(E_{c/r})}B_j' & 0
\end{bmatrix}  , \quad j=1, \ldots, \mu
\end{eqnarray*}
for
\[
A_i=F_{12}^{(i-1)} (F_{22}^{(i-1)})^{-1} N^{(E_{c/r})}, \quad B_j= (F_{22}^{(\mu+j-1)})^{-1}  F_{21}^{(\mu+j-1)}.
\]
\end{proposition}

\begin{proof}
The assertion follows from
\begin{eqnarray*}
K_{1-2}&:=&K_1 K_2=\begin{bmatrix}
	I_d & A_{1} \\
	0 & I_{a}
\end{bmatrix} \cdots \begin{bmatrix}
	I_d & A_{\mu} \\
	0 & I_{a}
\end{bmatrix}\cdot \begin{bmatrix}
	I_d & 0 \\
	B_{1} & I_{a}
\end{bmatrix} \cdots \begin{bmatrix}
	I_d & 0 \\
	B_{\mu} & I_{a}
\end{bmatrix} \\
&=& \begin{bmatrix}
	I_d & A \\
	0 & I_{a}
\end{bmatrix}\cdot \begin{bmatrix}
	I_d & 0 \\
	B & I_{a}
\end{bmatrix}=\begin{bmatrix}
	I_d+AB & A \\
	B & I_{a}
\end{bmatrix} 
\end{eqnarray*}
and
\[
K^{-1}_{1-2}:=K_2^{-1}K_1^{-1}= \begin{bmatrix}
	I_d & 0 \\
	-B & I_{a}
\end{bmatrix}  \begin{bmatrix}
	I_d & -A\\
	0 & I_{a}
\end{bmatrix}=\begin{bmatrix}
	I_d & -A\\
	-B & BA+ I_{a}
\end{bmatrix}.
\]
\end{proof}

Note that this Proposition enables us to compute in general
\[
\pPi_{can}= K_0 \begin{bmatrix}
	I_d + AB & -A-ABA \\
	B & -BA
\end{bmatrix}K_0^{-1}
\]
directly from $K_0$, $K_1$, $K_2$, since $A=(K_1)_{12}=(K_1K_2)_{12}$, $B=(K_2)_{21}=(K_1K_2)_{21}$. 
Moreover, if needed, a closer look provides the precise smoothness assumptions required for parts of $F_0$ that we will not go into here.

From the above proof we recognize that, as soon as we finish Step 1, a pure ODE can be represented, since it means that
for every pair
\[
\left\{E_1, F_1 \right\} =\left\{ \begin{bmatrix}
		I_d & 0\\
		0 & N^{(E_{r/c})}
	\end{bmatrix} , \begin{bmatrix}
		F_{11} & 0\\
		F_{21} & F_{22}
	\end{bmatrix}\right\},
\]
it holds
\[
\left\{E_1, F_1 \right\} \overset{L_2,K_2}{\sim} \left\{ \begin{bmatrix}
		I_d & 0\\
		0 & N^{(E_{r/c})}
	\end{bmatrix} , \begin{bmatrix}
		F_{11} & 0\\
		0 & F_{22}
	\end{bmatrix}\right\} = \left\{E_2, F_2 \right\}
\]
with
\[
L_2= \begin{bmatrix}
	I_d & 0 \\
	N^{(E_{c/r})} B & I_a
\end{bmatrix}, \quad 
K_2= \begin{bmatrix}
	I_d & 0 \\
	B & I_a
\end{bmatrix}.
\]

\begin{corollary}\label{cor:pureODE_01}
A pure ODE \eqref{eq:pureODE} is obtained for
\[
u(t):=\begin{bmatrix}
	I_d & 0
\end{bmatrix} (K_0K_1)^{-1}, \quad \bar{q}_1 := \begin{bmatrix}
	I_d & 0
\end{bmatrix} L_1L_0q
\]
and $\Omega$ resulting from Step 1.
\end{corollary}
\begin{proof}
The assertion follows from the particular form of $K_3, K_2$, $L_3, L_2$:
\[
\begin{bmatrix}
	I_d & 0
\end{bmatrix} K_3^{-1}K_2^{-1}K_1^{-1}K_0^{-1} =
\begin{bmatrix}
	I_d & 0
\end{bmatrix} K_1^{-1}K_0^{-1}, \quad \begin{bmatrix}
	I_d & 0
\end{bmatrix} L_3L_2L_1L_0 =
\begin{bmatrix}
	I_d & 0
\end{bmatrix} L_1L_0. 
\]
\end{proof}

\section{Application to T-canonical and S-canonical forms}\label{sec:T-S-canonical}

\subsection{T-canonical form}

Due to Lemma \ref{lem:Ncr} we can assume that $N$ has the required secondary diagonal blocks such that
according to Lemma \ref{lem:Rcr} 
\[
R^c:=N(N^{(E_c)})^T+(I-N^{(E_c)}(N^{(E_c)})^T) \in \BUT_{nonsingular}
\]
is nonsingular and $R^{-c}:=(R^c)^{-1}\in \BUT_{nonsingular}$ fulfills $R^{-c}N=N^{(E_c)}$.
Therefore, with
\[
L_c =\begin{bmatrix}
	I_d & \\
	 & R^{-c}
\end{bmatrix}, \quad K_c= \begin{bmatrix}
	I_d & \\
	 & I_{a}
\end{bmatrix}
\]
it follows
\[
\left\{ 
\begin{bmatrix}
	I_{d} & 0\\
	0 & N
\end{bmatrix}, \ \begin{bmatrix}
	\Omega & 0 \\
	F_{21} & I_{a}
\end{bmatrix}
\right\} \quad \sim \quad 
\left\{ 
\begin{bmatrix}
	I_{d} & 0\\
	0 & N^{(E_c)}
\end{bmatrix},  \begin{bmatrix}
	\Omega & 0  \\
	R^{-c}F_{12} & R^{-c}
\end{bmatrix}
\right\},
\]
that is in a PreSCF  according to Definition \ref{def:prel_SCF} for decreasing block sizes $\ell_1 \geq \cdots \geq \ell_{\mu} $. Note further that due to $F_{12}=0$, the transformations in Steps 1 and 2 do not change neither $\Omega$ nor $R^{-c}$, and Step 3 and 4 lead to
\[
\left\{ 
\begin{bmatrix}
	I_{d} & 0\\
	0 & N^{(E_c)}
\end{bmatrix}, 
\begin{bmatrix}
	\Omega &  0 \\
	0 & R^{-c}
\end{bmatrix}
\right\}  \quad \sim \quad  \left\{ \begin{bmatrix}
	I_{d} & 0\\
	0 & N^{(E_c)}
\end{bmatrix}, 
\begin{bmatrix}
	\Omega &   0\\
	0 & I_{a}
\end{bmatrix} 
\right\} .
\] 

Moreover, if $L_T, K_T$ are the transformation matrix functions into the T-canonical form, then due to $F_{12}=0$ also $A=0$ and
\[
\pPi_{can} = K_T \begin{bmatrix}
	I_d & 0 \\
	B & 0
\end{bmatrix} K_T^{-1}.
\]

\subsection{S-canonical form}
Due to Lemma \ref{lem:Ncr} we can assume that $N$ has the required secondary diagonal blocks such that
according to Lemma \ref{lem:Rcr} 
\[
R^r:=(N^{(E_r)})^TN+(I-(N^{(E_r)})^TN^{(E_r)}) \in \BUT_{nonsingular}
\]
is nonsingular and $R^{-r}:=(R^r)^{-1}\in \BUT_{nonsingular}$ fulfills $NR^{-r}=N^{(E_r)}$.
Therefore, with
\[
L_r =\begin{bmatrix}
	I_d & \\
	 & I_{a}
\end{bmatrix}, \quad K_r= \begin{bmatrix}
I_d & -E_{12}R^{-r}\\
0	 & R^{-r}
\end{bmatrix}
\]
it follows
\[
\left\{ 
\begin{bmatrix}
	I_{d} & E_{12}\\
	0 & N
\end{bmatrix},  \begin{bmatrix}
	\Omega & 0 \\
	0 & I_{a}
\end{bmatrix}
\right\} \sim  
\left\{ 
\begin{bmatrix}
	I_{d} & 0\\
	0 & N^{(E_r)}
\end{bmatrix}, \begin{bmatrix}
	\Omega & -\Omega E_{12}R^{-r} - (E_{12}R^{-r})'+E_{12}(R^{-r})'  \\
	0 & R^{-r} + N(R^{-r})'
\end{bmatrix}
\right\}
\]
that is in a PreSCF according to Definition \ref{def:prel_SCF} for increasing block sizes $\ell_1 \leq \cdots \leq \ell_{\mu} $.
Note further that due to $F_{21}=0$, the transformations in Steps 1 and 2 do not change neither $\Omega$ nor $R^{-r}+ N(R^{-r})'\in \BUT_{nonsingular}$, and Step 3 and 4 lead to
\[
\left\{ 
\begin{bmatrix}
	I_{d} & 0\\
	0 & N^{(E_r)}
\end{bmatrix}, 
\begin{bmatrix}
	\Omega &  0 \\
	0 & R^{-r}+ N(R^{-r})'
\end{bmatrix}
\right\} \sim \left\{ \begin{bmatrix}
	I_{d} & 0\\
	0 & N^{(E_r)}
\end{bmatrix}, 
\begin{bmatrix}
	\Omega &   0\\
	0 & I_{a}
\end{bmatrix} 
\right\} .
\] 
Moreover, if $L_S, K_S$ are the transformation matrix function into the S-canonical form, then due to $F_{21}=0$ also $B=0$ and
\[
\pPi_{can} = K_SK_r \begin{bmatrix}
	I_d & -A \\
	0 & 0
\end{bmatrix} K_r^{-1}K_S^{-1}.
\]

\section{A closer look to DAEs in Hessenberg form} \label{sec:Hessenberg}

\subsection{Definition and canonical characteristics}

Let us consider pairs $\{H_E, H_F\}$ in Hessenberg form for $\mu\geq 2$. This means that there exist $m_j \in \N$, $j=1, \ldots, \mu$, $~m=m_1 + \ldots + m_{\mu}$, $r=m_1 + \ldots + m_{\mu-1}$, $m_{\mu}=m-r$, matrix functions
\[
H_{ij} :\mathcal I\rightarrow \reals^{m_i \times m_j}
\]
and that the pattern
\begin{eqnarray}
H_E=\begin{bmatrix}
	I_r & 0 \\
	0 & 0
\end{bmatrix}, \quad H_F=
\begin{bmatrix}
   H_{11}&H_{12}&\cdots&H_{1, \mu-1}&H_{1 \mu} \\
	 H_{21}&H_{22}&\cdots&H_{2 ,\mu-1}&0 \\
   &\ddots&\vdots & \vdots & \vdots\\
  &&H_{\mu-1, \mu-2}&H_{\mu-1, \mu-1} &0 \\
	 &&&H_{\mu, \mu-1}  &0\\
   \end{bmatrix} \begin{array}{lll}
		\left.\right\}  ~m_1 \\
		\left.\right\}  ~m_2\\
		\quad  \vdots  \\
		\left.\right\} 	m_{\mu-1} \\
		\left.\right\} 	m_{\mu}&=&m-r
	 \end{array} 
\label{eq:HessenberForm}
\end{eqnarray}
with nonsingular
\[
H_{\mu, \mu-1} H_{\mu-1, \mu-2} \cdots H_{21} H_{1\mu}:\mathcal I\rightarrow \reals^{m_{\mu} \times m_{\mu}}
\]
is given for all $t \in \mathcal I$, such that $m_1 \geq m_2 \geq \cdots \geq m_{\mu}>0$.

Using \cite[Theorem 8.1]{commonground2024}, the rank properties proven in \cite[Theorem 3.42]{CRR} imply
\[
\theta_0= \cdots = \theta_{\mu-2} =m_{\mu}=m-r=:\theta, \quad d=r-(\mu-1)\theta = m-\mu\theta.
\]
In particular, this means that
\[
N^{(E)}:=\begin{bmatrix}
	0 & I_\theta &  & \\
   & \ddots & \ddots & \\
	& & 0 & I_\theta \\
		&&& 0
\end{bmatrix} \quad \mbox{with} \quad \ell_1 = \cdots = \ell_{\mu}=\theta=m-r
\]
can be interpreted as $N^{(E_c)}$ or $N^{(E_r)}$. Of course, $H_E$ can easily be transformed into $\begin{bmatrix}I_d & 0 \\ 0 &N^{(E)}\end{bmatrix}$. Below, if particular orthogonal $L_0$ and $K_0$ are used, then the matrix  function $F_0$ in  \eqref{eq:Initial_eq} presents a block structure that corresponds to a PreSCF.

\subsection{The index-2 case}
Let us first focus on $\mu=2$ and therefore
\begin{eqnarray}
H_E=\begin{bmatrix}
	I_r & 0 \\
	0 & 0
\end{bmatrix}, \quad H_F=
\begin{bmatrix}
   H_{11}&H_{12}\\
	 H_{21}  &0\\
   \end{bmatrix} \begin{array}{lll}
		\left.\right\}  ~m_1 \\
		\left.\right\} 	~m_{2}&=&m-r = \theta
	 \end{array} 
\label{eq:HessenberFormI2}
\end{eqnarray}
with nonsingular
\[
H_{21} H_{12}:\mathcal I\rightarrow \reals^{m_{2} \times m_{2}}
\]
is given for all $t \in \mathcal I$.

\subsubsection{Step 0: An initial orthogonal equivalence transformation into PreSCF}
To decouple equations presenting this structure, 
we  use a matrix function $B_{d,2}$ whose columns are an orthonormal basis  of $\ker H_{21}$ and a matrix function $B_{a,2}$ whose columns are an orthonormal basis  of the orthogonal complement, such that
\[
B_{d,2} : \mathcal I \rightarrow \reals^{m_1 \times (m_1-\theta)}, \quad B_{a,2} : \mathcal I \rightarrow \reals^{m_1 \times \theta}, \quad \im B_{d,2} \oplus \im B_{a,2}= \reals^{m_1},
\]
and
\[
	B_{d,2}^TB_{d,2}=I_{m_1-\theta},\quad B_{a,2}^TB_{a,2}=I_{\theta}, \quad B_{d,2}^TB_{a,2}=0.
\] 
With these rectangular matrices functions, we construct the orthogonal matrix functions
\[
L_{orth}= \begin{bmatrix}
	B_{d,2}^T & 0 \\
	B_{a,2}^T & 0 \\
	0 & I_{\theta}
\end{bmatrix}, \quad K_{orth}=\begin{bmatrix}
	B_{d,2} & 0 & B_{a,2}\\
	0 & I_{\theta} & 0 
\end{bmatrix},
\]
such that $\{ H_E,H_F\} \sim \{ E_0,F_0\}$ with
\[
E_0 = L_{orth} H_E K_{orth} = \left[\begin{array}{c|cc}
	I_{m_1-\theta} & 0 & 0 \\
	\hline
	0 & 0 & I_{\theta} \\
	0 & 0 & 0 
\end{array}\right] =\begin{bmatrix}
	I_d & 0 \\
	0 & N^{(E)}
\end{bmatrix},
\]

\[
 L_{orth} H_E K_{orth}' = \left[\begin{array}{c|cc}
	B_{d,2}^TB_{d,2}'   & 0 & B_{d,2}^TB_{a,2} \\
	\hline
	B_{a,2}^TB_{d,2} & 0 & B_{a,2}^TB_{a,2}' \\
	0 & 0 & 0 
\end{array}\right],
\]
\[
F_0= L_{orth} H_F K_{orth} +L_{orth} H_E K_{orth}' = \begin{bmatrix}
	F_{11} & F_{12}\\
	F_{21} & F_{22}
\end{bmatrix}\begin{matrix}
\left.\right\} & d \\
\left.\right\} & a
\end{matrix}
\]
with $d=m-2\theta=m_1-m_2$,  $F_{11}:\mathcal I\rightarrow \reals^{d \times d}$ and
\[
F_{22}=\begin{bmatrix}
   (F_{22})_{11}&(F_{22})_{12} \\
  0 &(F_{22})_{22}\\
   \end{bmatrix} \begin{array}{lll}
		\left.\right\}  ~\ell_1&=& \theta \\
		\left.\right\}  ~\ell_2&=&\theta=m_2=m-r
	 \end{array} 
\]
having quadratic blocks
\[
(F_{22})_{11}=B_{a,2}^TH_{12}:\mathcal I\rightarrow \reals^{\theta \times \theta},
\]
\[
(F_{22})_{22}=H_{21}B_{a,2}:\mathcal I\rightarrow \reals^{\theta \times \theta}.
\]
These blocks are nonsingular, since
\begin{itemize}
	\item $H_{12}$ has full column rank $m_2$, $B_{a,2}^T$ has full row rank $m_2$, $m_2=\theta$.
	\item $H_{21}$ has full row rank $m_2$, $B_ {a,2}$ has full column rank $m_2$, $m_2=\theta$,
	\item $\theta=m_2=\rank H_{21}H_{12} =  \rank H_{21}B_ {a,2}B_ {a,2}^TH_{12} \leq \min \left\{ \rank H_{21}B_ {a,2}, \rank B_ {a,2}^TH_{12}\right\} \leq \theta$, whereas we used that $B_ {a,2}B_ {a,2}^T$ is an orthogonal projector onto $(\ker H_{21})^{\perp}$.
\end{itemize}

Moreover, we have the structure
\[
F_{12}=\begin{blockarray}{ccc}
     \ell_1 & \ell_2  \\
\begin{block}{[cc]c}
 * & * & m_1-\theta  \\
\end{block}
\end{blockarray},
\]
as well as
\[
F_{21}=\begin{blockarray}{cc}
    m_1-\theta   &  \\
\begin{block}{[c]c}
  * & \ell_1\\
 0  &  \ell_{2}\\
\end{block}
\end{blockarray},
\]
whereas $*$ stands for possibly nonzero and time varying  blocks.

Because of this pattern, we have the block structure
\[
F_{21}F_{12} =\begin{blockarray}{ccc}
    \ell_1   &\ell_{2} &  \\
\begin{block}{[cc]c}
 * & *& \ell_1 =\theta\\
	0 & 0 &   \ell_{2}=\theta  \\
\end{block}
\end{blockarray},
\]
such that a PreSCF is given. Therefore,
$F_{21}F_{12}F_{22}^{-1}N^{(E)}$ and $N^{(E)} F_{22}^{-1}F_{21}F_{12}$ have the strictly upper block structure
\[
\begin{blockarray}{ccc}
    \ell_1   &\ell_{2} &  \\
\begin{block}{[cc]c}
 0 & *&  \ell_1 =\theta\\
	0 & 0 &  \ell_{2}=\theta  \\
\end{block}
\end{blockarray},
\]
as required for Steps 1 and 2.

\subsubsection{An index-2 example} \label{sec:Ex_HMM98}
\begin{example}\cite[Section 5]{HaMaMae98}, \cite[Example 8.11]{CRR} \label{ex:HMM98}
 Hessenberg form, index $\mu=2$, $m=3$, $r=2$, $\theta_0=1$, $d=1$.
\[
E=
\left[\begin{matrix}1 & 0 & 0\\0 & 1 & 0\\0 & 0 & 0\end{matrix}\right]
,\quad 
F=
\left[\begin{matrix}\lambda & -1 & -1\\\eta \left(- \eta t^{2} + t - 1\right) & \lambda & - \eta t\\- \eta t + 1 & 1 & 0\end{matrix}\right].
\]
For shortness, let us define $\gamma(t):= \sqrt{\left(-\eta t + 1\right)^{2} + 1} \neq 0$.
\end{example}
Step 0: $\{ E,F\} \sim \{ E_0,F_0\}$
\[
L_0=
\left[\begin{matrix}\frac{1}{\gamma(t)} & \frac{\eta t - 1}{\gamma(t)} & 0\\\frac{- \eta t + 1}{\gamma(t)} & \frac{1}{\gamma(t)} & 0\\0 & 0 & 1\end{matrix}\right]
,\quad 
K_0=
\left[\begin{matrix}\frac{1}{\gamma(t)} & 0 & \frac{- \eta t + 1}{\gamma(t)}\\\frac{\eta t - 1}{\gamma(t)} & 0 & \frac{1}{\gamma(t)}\\0 & 1 & 0\end{matrix}\right],
\]

\[
E_0=
\left[\begin{array}{c|cc} 1 & 0 & 0\\
\hline
0 & 0 & 1\\0 & 0 & 0\end{array}\right]
=\begin{bmatrix}
	I_d & 0 \\
	0 & N^{(E)}
\end{bmatrix},
\]
\[
F_0=
\left[\begin{array}{c|cc}
\frac{- \eta \left(\eta t - 1\right) \left(t \left(\eta t - 1\right) + 1\right) + \lambda + \left(\eta t - 1\right) \left(\lambda \left(\eta t - 1\right) - 1\right)}{\gamma(t)^2} & \frac{- \eta t \left(\eta t - 1\right) - 1}{\gamma(t)} & \frac{\eta^{4} t^{4} - 3 \eta^{3} t^{3} + \eta^{3} t^{2} + 3 \eta^{2} t^{2} - 2 \eta^{2} t - \eta t - 1}{\gamma(t)^2}\\
\hline
\frac{- \eta t + 1}{\gamma(t)^2} & - \frac{1}{\gamma(t)} & \frac{\eta t + \lambda + \left(\eta t - 1\right) \left(\eta t \left(\eta t - 1\right) + \eta + \lambda \left(\eta t - 1\right)\right) - 1}{\gamma(t)^2}\\0 & 0 & \gamma(t)
\end{array}\right].
\]
As expected, by construction $F_{22}$ is nonsingular and (block) triangular. Now we show how starting from $\{E_0,F_0\}$ an SCF and an SSCF can be constructed.

Step 1: $\{ E_0,F_0\} \sim \{ E_1,F_1\}$ with $E_1=E_0$ and
\[
L_1=
\left[\begin{array}{c|cc}1 & - \eta^{2} t^{2} + \eta t - 1 & \eta t \gamma(t)\\
\hline 0 & 1 & 0\\0 & 0 & 1\end{array}\right]
,\quad 
K_1=
\left[\begin{array}{c|cc}1 & 0 & \eta^{2} t^{2} - \eta t + 1\\
\hline
0 & 1 & 0\\0 & 0 & 1\end{array}\right],
\]
\[
F_1=
\left[\begin{array}{c|cc}\omega(t) & 0 & 0\\
\hline
\frac{- \eta t + 1}{\gamma(t)^2} & - \frac{1}{\gamma(t)} & \frac{\eta^{2} \lambda t^{2} + \eta^{2} t - 2 \eta \lambda t - \eta + 2 \lambda}{\gamma(t)^2}\\0 & 0 & \gamma(t)\end{array}\right],
\]
for
\[
\omega(t):=\frac{\eta^{2} \lambda t^{2} - \eta^{2} t - 2 \eta \lambda t + \eta + 2 \lambda}{\gamma(t)^2} = \lambda + \frac{\ - \eta^{2} t  + \eta }{\gamma(t)^2}.
\]
Step 2: $\{ E_1,F_1\} \sim \{ E_2,F_2\}$ with $E_2=E_1=E_0$ and
\[
L_2=
\left[\begin{array}{c|cc}1 & 0 & 0\\
\hline 
0 & 1 & 0\\0 & 0 & 1\end{array}\right]
,\quad 
K_2=
\left[\begin{array}{c|cc}1 & 0 & 0\\
\hline
\frac{- \eta t + 1}{\gamma(t)} & 1 & 0\\0 & 0 & 1
\end{array}\right],
\]
\[
F_2=
\left[\begin{array}{c|cc}
\omega(t) & 0 & 0\\
\hline 0 & - \frac{1}{\gamma(t)} & \frac{\eta^{2} \lambda t^{2} + \eta^{2} t - 2 \eta \lambda t - \eta + 2 \lambda}{\gamma(t)^2}\\0 & 0 & \gamma(t)
\end{array}\right].
\]

Step 3: $\{ E_2,F_2\} \sim \{ E_{SCF},F_{SCF}\}$
\[
L_3=\begin{bmatrix}
	1 & 0 \\
	0 & (F_2)_{22}^{-1}
\end{bmatrix}, \quad K_3=I_3,
\]
\[
E_{SCF}=
\left[\begin{array}{c|cc} 1 & 0 & 0\\
\hline
0 & 0 & - \gamma(t)\\0 & 0 & 0\end{array}\right]
, \quad F_{SCF} = \left[\begin{array}{c|cc} \omega(t) & 0 & 0\\
\hline
0 & 1 & 0\\0 & 0 & 1\end{array}\right].
\]
 Step 4: Following the procedure from \cite{SSCF25}, with
\[
L_4=\left[\begin{array}{c|cc}
	1 & 0 & 0  \\
	\hline
	0 & \frac{1}{-\gamma(t)} & 0 \\
	0 & 0 & 1
\end{array}\right]
, \quad K_4=\left[\begin{array}{c|cc}
	1 & 0 & 0  \\
	\hline
	0 & -\gamma(t)& 0 \\
	0 & 0 & 1
\end{array}\right].
\]
we finally obtain
\[
\{ E_{SCF},F_{SCF}\}\sim \{ E_{SSCF},F_{SSCF}\}
\]
with
\[
E_{SSCF}=\left[\begin{array}{c|cc}
	1 & 0 & 0  \\
	 \hline
	0 & 0 & 1 \\
	0 & 0 & 0
\end{array}\right], \quad F_{SSCF}=F_{SCF}=\left[\begin{array}{c|cc} \omega(t) & 0 & 0\\
\hline
0 & 1 & 0\\0 & 0 & 1\end{array}\right].
\]
Note that the existence of an SCF and an SSCF for the same DAE corresponds to the uniqueness properties of the SCF, see Theorem \ref{th:Uniqueness_SCF} in the Appendix \ref{sec:Uniqueness}.

Since
\[
L:=L_2L_1 L_0=
\left[\begin{matrix}\eta t \gamma(t) & \frac{- \eta^{2} t^{2} + 2 \eta t - 2}{\gamma(t)} & \eta t \gamma(t)\\\frac{- \eta t + 1}{\gamma(t)} & \frac{1}{\gamma(t)} & 0\\0 & 0 & 1\end{matrix}\right]
,\]
and
\[
K:=K_0K_1K_2=
\left[\begin{matrix}\frac{1}{\gamma(t)} & 0 & \gamma(t)\\\frac{\eta t - 1}{\gamma(t)} & 0 & \eta t \gamma(t)\\\frac{- \eta t + 1}{\gamma(t)} & 1 & 0\end{matrix}\right]
\]
provide $\{E_2, F_2\}$, with
\[
K^{-1}=
\left[\begin{matrix}\eta t \gamma(t) & - \gamma(t) & 0\\\eta t \left(\eta t - 1\right) & - \eta t + 1 & 1\\\frac{- \eta t + 1}{\gamma(t)} & \frac{1}{\gamma(t)} & 0\end{matrix}\right],
\]
we can represent  the canonical projector
\begin{align}
\pPi_{can}= K \begin{bmatrix}
 1 & 0 & 0 \\
0 & 0 & 0\\
 0 & 0& 0
\end{bmatrix} K^{-1}=
\left[\begin{matrix}\eta t & -1 & 0\\\eta t \left(\eta t - 1\right) & - \eta t + 1 & 0\\\eta t \left(- \eta t + 1\right) & \eta t - 1 & 0\end{matrix}\right]. \label{eq:PicanHMM}
\end{align}

Alternatively, we can compute $\pPi_{can}$ with Proposition \ref{prop:Pican}. From the representation of $K_1$ and $K_2$ we know
\[
A=\begin{bmatrix}
	0 & \eta^{2} t^{2} - \eta t + 1
\end{bmatrix}, \quad B= \begin{bmatrix}
	\frac{- \eta t + 1}{\gamma(t)} \\
	0
\end{bmatrix},
\]
such that $AB=0$, therefore also $ABA=0$, and
\[
\pPi_{can}= K_0 \begin{bmatrix}
	I_d  & -A \\
	B & -BA
\end{bmatrix}K_0^{-1}= K_0\begin{bmatrix}
	1  & 0 & -\eta^{2} t^{2} + \eta t - 1 \\
	\frac{- \eta t + 1}{\gamma(t)}  & 0 & \frac{(- \eta t + 1)(\eta^{2} t^{2} - \eta t + 1)}{\gamma(t)} \\
	0 & 0 & 0
\end{bmatrix}K_0^{-1},
\]
that also leads to the expression \eqref{eq:PicanHMM}.

Let us finally  compute the solution of the resulting homogeneous DAE $E_{SSCF}z'+F_{SSCF}z=0$ for
$z = K^{-1}x$. 
The solution of the pure ODE
\[
z_1'=-\omega(t)z_1
\]
reads
\[
z_1(t) = C \cdot e^{ -\lambda t} \cdot \gamma(t).
\]
and the pure DAE leads to
\[
z_3=0, \quad z_2=0,
\]
such that
\begin{eqnarray*}
x&=&Kz=\left[\begin{matrix}\frac{1}{\gamma(t)} & 0 & \gamma(t)\\\frac{\eta t - 1}{\gamma(t)} & 0 & \eta t \gamma(t)\\\frac{- \eta t + 1}{\gamma(t)} & 1 & 0\end{matrix}\right]
 \begin{bmatrix}
C \cdot e^{ -\lambda t} \cdot \gamma(t)
 \\
0 \\
0
\end{bmatrix} = C \cdot e^{ -\lambda t}
\begin{bmatrix}
1 \\
\eta t -1 \\
-\eta t +1
\end{bmatrix}.
\end{eqnarray*}
This coincides with the known solution from literature. However, note that the pure ODE obtained in \cite{HaMaMae98} from an IERODE is $\tilde{z}_1'=-\lambda ~\tilde{z}_1$. This apparent difference results from the choice of $K_0$ only and corresponds to $z_1=  \gamma ~\tilde{z}_1$, since also the pure ODE is not unique.
Indeed, in terms of Corollary \ref{cor:uniqSCF}, for
\[
\bar L =\begin{bmatrix}
	-\gamma(t) & 0 \\
	0 & I_2
\end{bmatrix}, \quad \bar K =\begin{bmatrix}
\frac{1}{-\gamma(t)}	& 0 \\      
	0 & I_2
\end{bmatrix}
\] 
it holds
\[
\{E_{SSCF}, F_{SSCF}\} \quad  \overset{\bar{L}, \bar{K}}{\sim} \quad \left\{ \left[\begin{array}{c|cc}
	1 & 0 & 0  \\
	 \hline
	0 & 0 & 1 \\
	0 & 0 & 0
\end{array}\right], \left[\begin{array}{c|cc} \lambda & 0 & 0\\
\hline
0 & 1 & 0\\0 & 0 & 1\end{array}\right]\right\}.
\]

We want to emphasize also that choosing at the beginning
\[
\tilde{L}_0=
\left[\begin{matrix}\frac{1}{\gamma(t)^2} & \frac{\eta t - 1}{\gamma(t)^2} & 0\\\frac{- \eta t + 1}{\gamma(t)^2} & \frac{1}{\gamma(t)^2} & 0\\0 & 0 & 1\end{matrix}\right]
,\quad 
\tilde{K}_0=
\left[\begin{matrix}1 & 0 & - \eta t + 1\\\eta t - 1 & 0 & 1\\0 & 1 & 0\end{matrix}\right]
\]
Steps 0, 1, 2 lead to the  pure ODE $\tilde{z}_1'=-\lambda \tilde{z}_1$ straight forward. This makes clear that, although in general we proposed to choose orthogonal $L_0$ and $K_0$, a different scaling may be reasonable to define $K_0$  as simple as possible.

\subsection{The particular structure resulting in Multibody Dynamics}
\subsubsection{A Hessenberg form of order 3}
Considering linear (or linearized) constrained multibody systems with $p$ and $v$ describing the $n_p$ position and velocity coordinates, respectively, leads to equations of the form
\[
\begin{bmatrix}
	I & 0 & 0 \\
	0 & M & 0 \\
	0 & 0 & 0
\end{bmatrix}
\begin{bmatrix}
	p'\\
	v'\\
	0
\end{bmatrix} + \begin{bmatrix}
	0 & -I & 0 \\	
	K & D & G^T \\
	G & 0 & 0 
\end{bmatrix}\begin{bmatrix}
	p\\
	v\\
	\lambda
\end{bmatrix} = q,
\]
where $G$ is the $n_{\lambda} \times n_{p}$ constraint matrix, $K,D,M$ are the stiffness-, damping- and mass matrices,  $q$ may contain time dependent excitations and the $n_{\lambda}$ unknowns $\lambda$ are the so-called Lagrange multipliers \cite{Eich-SoelFueh98}. For constant $G, K, D, M$, the state space form is well-understood. Here, we focus on the case that these matrices may vary with time, although we drop the argument $t$ for the sake of clarity. We start permuting and scaling the equations with
\[
L_H = \begin{bmatrix}
	0 & I_{n_p} & 0 \\
	M^{-1} & 0 & 0\\
	0 & 0 & I_{n_{\lambda}}
\end{bmatrix}, \quad K_H=\begin{bmatrix}
	0 & I_{n_p} & 0 \\
	I_{n_p} & 0 & 0\\
	0 & 0 & I_{n_{\lambda}}
\end{bmatrix}
\]
to obtain the pair in Hessenberg form
\[
\{H_E, H_F\}=\left\{ \begin{bmatrix}
	I & 0 & 0 \\
	0  & I & 0 \\
	0 & 0 & 0
\end{bmatrix}, \begin{bmatrix}
	M^{-1}D & M^{-1}K & M^{-1}G^T \\
	-I & 0 & 0 \\
	0 &  G & 0
\end{bmatrix}\right\},
\]
with nonsingular $GM^{-1}G^T$, $m_1=m_2=n_p$, $m_3=n_{\lambda}=\theta$, $x_1=v$, $x_2=p$, $x_3=\lambda$. 

We also want to emphasize that if in the 3D-case the angular velocities are considered as well, then in the linear (or linearized) case we consider
\[
x_1=\begin{bmatrix}
	v\\
	s
\end{bmatrix}, \quad x_2=p, \quad x_3 =\lambda, \quad m_1=n_p+n_s, \ m_2=n_p,  \ m_3=n_{\lambda}=\theta,
\] 
and $p'=Zv$ for a regular transformation matrix function $Z$, leading to
\[
H_{13}=\begin{bmatrix}
	M^{-1}Z^TG^T\\
	0
\end{bmatrix}, \quad H_{21}= \begin{bmatrix}
	-Z & 0
\end{bmatrix}, \quad H_{32} = G,
\]
see \cite{Eich-SoelFueh98}, \cite{CRR}, with nonsingular
\[
H_{32} H_{21}H_{13} = -GZM^{-1}Z^TG^T.
\]
Therefore,  a special Hessenberg form
\[
\left\{ \begin{bmatrix}
	I_{m_1} & 0 & 0 \\
	0  & I_{m_2} & 0 \\
	0 & 0 & 0
\end{bmatrix}, \begin{bmatrix}
	H_{11} & H_{12}  & H_{13} \\
	[Z_{21} \ 0]& 0& 0 \\
	0 &  G & 0
\end{bmatrix}\right\}
\]
for a quadratic, nonsingular $Z_{21}$ results. Notice now that the equivalence transformation with
\[
L_Z = \begin{bmatrix}
	I_{m_1}& 0  & 0 \\
	0 & Z_{21}^{-1} & 0 \\
	0 & 0 & I_{m_3}
\end{bmatrix}, \quad K_Z=
\begin{bmatrix}
	I_{m_1} & 0 & 0 \\
	0 & Z_{21} & 0 \\
	0 & 0 & I_{m_3}
\end{bmatrix}
\]
leads to
\[
\left\{ \begin{bmatrix}
	I_{m_1} & 0 & 0 \\
	0  & I_{m_2} & 0 \\
	0 & 0 & 0
\end{bmatrix}, \begin{bmatrix}
	H_{11} & H_{12}Z_{21}  & H_{13} \\
	[I_{m_2} \ 0]& Z_{21}^{-1}Z_{21}' & 0 \\
	0 &  H_{32}Z_{21} & 0
\end{bmatrix}\right\}
\]
with nonsingular
\[
H_{32}Z_{21}[I_{m_2} \ 0]H_{13}=-GZM^{-1}Z^TG^T.
\]

Therefore, for the sake of simplicity in the following we assume that the particular Hessenberg structure
\begin{eqnarray}
\{H_E, H_F\}=\left\{ \begin{bmatrix}
	I_{m_1} & 0 & 0 \\
	0  & I_{m_2} & 0 \\
	0 & 0 & 0
\end{bmatrix}, \begin{bmatrix}
	H_{11} & H_{12}  & H_{13} \\
	[I_{m_2} \ 0]& H_{22} & 0 \\
	0 &  H_{32} & 0
\end{bmatrix}\right\}
\label{eq:Hessenberg3I}
\end{eqnarray}
with nonsingular  $H_{32}[I_{m_2}\ 0]H_{13}$  for all $t \in \mathcal I$ is given, implying $\rank H_{32}=m_3=\theta$, $m=m_1+m_2+m_3$, $r=m_1+m_2$. Of course, for  $m_1=m_2$ we simply assume $[I_{m_2} 0]=I_{m_2}=I_{m_1}$.
\medskip

Finally, we would like to note that the  transformations from below could also be generalized for Hessenberg forms with $\mu>3$ assuming 
\begin{eqnarray*}
H_{i+1,i}=[I_{m_{i+1}} \ 0] \quad  \mbox{for}  \quad i=1, \ldots, \mu-2
\label{eq:Assumption_H_I}.
\end{eqnarray*}

\subsubsection{Step 0: An initial orthogonal equivalence transformation into PreSCF } \label{sec:Step0}

To decouple pairs presenting the structure \eqref{eq:Hessenberg3I}, 
\begin{itemize}
	\item 
we first use a matrix function $B_{d,3}$ whose columns are an orthonormal basis  of $\ker H_{32}$ and a matrix function $B_{a,3}$ whose columns are an orthonormal basis  of the orthogonal complement, such that
\[
B_{d,3} : \mathcal I \rightarrow \reals^{m_2 \times (m_2-\theta)}, \quad B_{a,3} : \mathcal I \rightarrow \reals^{m_2 \times \theta}, \quad \im B_{d,3} \oplus \im B_{a,3}= \reals^{m_2},
\]
and
\[
B_{d,3}^TB_{d,3}=I_{m_2-\theta},\quad B_{a,3}^TB_{a,3}=I_{\theta}, \quad B_{d,3}^TB_{a,3}=0,
\] 
\item we secondly use matrix functions
\[
B_{d,2}=\begin{bmatrix}
	B_{d,3}  & 0 \\
	0 & I_{m_1-m_2}
\end{bmatrix}: \mathcal I \rightarrow \reals^{m_1 \times (m_1-\theta)}, \quad B_{a,2}=\begin{bmatrix}
	B_{a,3} \\
	0 
\end{bmatrix} : \mathcal I \rightarrow \reals^{m_1 \times \theta},
\]
fulfilling $\ker H_{32}H_{21}=\ker H_{32}[I_{m_2} \ 0] = \im  B_{d,2}$, $\im B_{d,2} \oplus \im B_{a,2}= \reals^{m_1},$
\[
 B_{d,2}^TB_{d,2}=I_{m_1-\theta},\quad B_{a,2}^TB_{a,2}=I_{\theta}, \quad B_{d,2}^TB_{a,2}=0,
\]
\item finally, since $\ker H_{32}H_{21}H_{13} = \{0\}$,  we set
$B_{a,1}=I_{m_3}$.
\end{itemize}

With this notation, we define the quadratic and nonsingular 
\[
L_B=\begin{blockarray}{cccc}
    m_1   &m_{2} & m_{3}& \\
\begin{block}{[ccc]c}
	B_{d,2}^T & 0 &0 & m_1-\theta\\
	0 & B_{d,3}^T & 0 & m_2 - \theta\\
	B_{a,2}^T & 0 & 0 & \theta\\
	0 & B_{a,3}^T & 0 & \theta \\
	0 & 0 & I_{m_3} & m_3=\theta \\
\end{block}
\end{blockarray}, \quad
K_B= \begin{blockarray}{cccccc}
    m_1-\theta   &m_{2}-\theta & m_{3}& \theta& \theta & \\
\begin{block}{[ccccc]c}
	B_{d,2} & 0 & 0 & B_{a,2} & 0& m_1\\
	0 & B_{d,3} & 0 &0 & B_{a,3} & m_2\\
	0 & 0 & I_{m_3} & 0& 0 & m_3=\theta \\
\end{block}
\end{blockarray}
.
\]
For \eqref{eq:Hessenberg3I} with $d=m-3 \theta=m_1+m_2-2m_3$, $L_0=L_B$ and $K_0=K_B$ leads to
\[
L_0 H_E K_0 = \begin{bmatrix}
	I_{d} & 0 \\
	0 & N^{(E)}
\end{bmatrix}, \quad N^{(E)}=\begin{bmatrix}
	0 & I_{\theta} & 0 \\
	0 & 0 & I_{\theta} \\
	0 & 0 & 0
\end{bmatrix} \in \reals^{a \times a}.
\]

\[
L_0 H_E K_0' =
\begin{blockarray}{cccccc}
    m_1-\theta   &m_{2}-\theta & \theta& \theta& \theta & \\
\begin{block}{[cc|ccc]c}
	B_{d,2}^TB_{d,2}' & 0 & 0 & B_{d,2}^T B_{a,2}' & 0& m_1-\theta\\
	0 & B_{d,3}^TB_{d,3}' & 0 &0 & B_{d,3}^TB_{a,3}' & m_2-\theta \\
	B_{a,2}^TB_{d,2}' & 0 & 0 & B_{a,2}^TB_{a,2}' & 0& \theta\\
	0 & B_{a,3}^TB_{d,3}' & 0 &0 & B_{a,3}^TB_{a,3}' & \theta \\
	0 & 0 & 0 & 0& 0 & \theta=m_3 \\
\end{block}
\end{blockarray},
\]
 and
\[
F_0:=L_0H_FK_0+L_0H_EK_0'=:\begin{bmatrix}
F_{11} & F_{12} \\
F_{21} & F_{22}
\end{bmatrix}\begin{matrix}
\left.\right\} & d \\
\left.\right\} & a
\end{matrix}
\]
with $d=m-3\theta=m_1+m_2-2m_3$,  $F_{11}:\mathcal I\rightarrow \reals^{d \times d}$ and
\[
F_{22}=\begin{bmatrix}
   (F_{22})_{11}&(F_{22})_{12}&(F_{22})_{13} \\
  0 &(F_{22})_{22}&(F_{22})_{23}\\
 0 &0 &(F_{22})_{33}
   \end{bmatrix} \begin{array}{lll}
		\left.\right\}  ~\ell_1&=& \theta \\
		\left.\right\}  ~\ell_2&=&\theta\\
		\left.\right\} 	~\ell_{3}&=&\theta=m_3=m-r
	 \end{array} 
\]
having quadratic blocks
\begin{eqnarray*}
(F_{22})_{11}&=&B_{a,2}^TH_{13}:\mathcal I\rightarrow \reals^{\theta \times \theta},\\
(F_{22})_{22}&=&B_{a,3}^T[I_{m_2} \ 0 ]B_{a,2} =I_{\ell_2}=I_{\theta}, \\
 (F_{22})_{33}&=&H_{32}B_{a,3}:\mathcal I\rightarrow \reals^{\theta \times \theta},
\end{eqnarray*}
that are nonsingular, since
\begin{itemize}
	\item $H_{13}$ has full column rank $m_3$, $B_{a,2}^T$ has full row rank  $m_3=\theta$,
	\item $H_{32}$ has full row rank $m_2$, $B_{a,3}$ has full column rank  $m_3=\theta$,
	\item $H_{32}H_{21}H_{13}=H_{32}B_ {a,3}B_ {a,3}^T\begin{bmatrix}
		I_{m_2} & 0
	\end{bmatrix}H_{13}=H_{32}B_ {a,3}B_ {a,2}^TH_{13}$, and therefore
	\[
	\theta = m_3=\rank H_{32}H_{21}H_{13}  \leq \min \left\{ \rank H_{32}B_ {a,3}, \rank B_ {a,2}^TH_{13}\right\} \leq \theta,
	\]
	whereas we used that $B_ {a,3}B_ {a,3}^T$ is an orthogonal projector onto $(\ker H_{32})^{\perp}$.
\end{itemize}
Moreover, we have the structure
\[
F_{12}=\begin{blockarray}{cccc}
     \ell_1 & \ell_2 &  \ell_{3}& \\
\begin{block}{[ccc]c}
 * & *& *  & m_1-\theta  \\
	0 & *& * &  m_{2}-\theta \\
\end{block}
\end{blockarray},
\]
as well as
\[
F_{21}=\begin{blockarray}{ccc}
    m_1-\theta   &m_{2}-\theta &  \\
\begin{block}{[cc]c}
 * & * & \ell_1 =\theta\\
	0 & *  &  \ell_{2}=\theta\\
	0 & 0 &   \ell_{3}=\theta \\
\end{block}
\end{blockarray},
\]
whereas $*$ stands for possibly nonzero and time varying  blocks.

Because of this pattern, we have the block structure
\[
F_{21}F_{12} =\begin{blockarray}{cccc}
    \ell_1   &\ell_{2} & \ell_{3}& \\
\begin{block}{[ccc]c}
 * & *& * & \ell_1 =\theta\\
	0 & * & * &  \ell_{2}=\theta \\
	0 & 0 &  0 & \ell_{3} =\theta \\
\end{block}
\end{blockarray},
\]
such that a PreSCF os given. Therefore,
$F_{21}F_{12}F_{22}^{-1}N^{(E)}$ and $N^{(E)} F_{22}^{-1}F_{21}F_{12}$ have the strictly upper block structure
\[
\begin{blockarray}{cccc}
    \ell_1   &\ell_{2} & \ell_{3}& \\
\begin{block}{[ccc]c}
 0 & *& * & \ell_1 =\theta\\
	0 & 0 & * &  \ell_{2}=\theta \\
	0 & 0 &  0 & \ell_{3}=\theta  \\
\end{block}
\end{blockarray},
\]
that is crucial for Steps 1 and 2.
Recall that for arbitrary $\{E,F\}$, the transformation matrix functions $L_0$ and $K_0$ to obtain the block structure $\{E_0,F_0\}$  may be composed step by step. Indeed, for multibody systems, we showed
\[
L_0=L_BL_ZL_H, \quad K_0= K_HK_ZK_B.
\]

\section{An illustrative index-3 example: The linearized Campbell–Moore DAE}\label{sec:Campbell-Moore}

The following example is a linearised version of a test problem proposed in \cite{CampbellMoore95} and has been discussed extensively in \cite{HaMae2023}.

\begin{example} \label{Ex_Cam_Moo}Semi-explicit, index $\mu=3$, $m=7$, $r=6$, $\theta_0=\theta_1=1$, $d=4$, assuming  $\alpha \neq 0$.

\[
E=
\left[\begin{matrix}I_6 & 0 \\ 0 & 0 \end{matrix}\right]
,\]
\[
F=
\left[\begin{matrix}0 & 0 & 0 & -1 & 0 & 0 & 0\\0 & 0 & 0 & 0 & -1 & 0 & 0\\0 & 0 & 0 & 0 & 0 & -1 & 0\\0 & 0 & \sin{\left(t \right)} & 0 & 1 & - \cos{\left(t \right)} & - \alpha \cos^{2}{\left(t \right)}\\0 & 0 & - \cos{\left(t \right)} & -1 & 0 & - \sin{\left(t \right)} & - \alpha \sin{\left(t \right)} \cos{\left(t \right)}\\0 & 0 & 1 & 0 & 0 & 0 & \alpha \sin{\left(t \right)}\\\alpha \cos^{2}{\left(t \right)} & \alpha \sin{\left(t \right)} \cos{\left(t \right)} & - \alpha \sin{\left(t \right)} & 0 & 0 & 0 & 0\end{matrix}\right].
\]
\end{example}
This DAE is not directly in Hessenberg form. To decouple it step by step, we use several equivalence transformations
\[
\{E,F\} \quad \sim \quad  \{LEK,LFK+LEK'\}.
\]
\subsection{Permutation into Hessenberg form}

With the block permutation matrix
\[
P=\begin{bmatrix}
	0 & I_3 & 0 \\
	I_3 & 0 & 0 \\
	0 & 0 & 1
\end{bmatrix}
\]
we obtain a pair in Hessenberg form $\{H_E, H_F\}$ for $L_H=K_H=P=P^T$, such that
\[
\{E,F\} \sim \{H_E, H_F\}
\]
and 
\[
H_E=L_H R K_H=E=\left[\begin{matrix}I_6 & 0 \\ 0 & 0\end{matrix}\right], \quad
H_F=L_HFK_H
,\] 
\[ 
H_F=
\left[\begin{array}{ccc|ccc|c}
0 & 1 & - \cos{\left(t \right)} & 0 & 0 & \sin{\left(t \right)} & - \alpha \cos^{2}{\left(t \right)}\\
-1 & 0 & - \sin{\left(t \right)} & 0 & 0 & - \cos{\left(t \right)} & - \alpha \sin{\left(t \right)} \cos{\left(t \right)}\\
0 & 0 & 0 & 0 & 0 & 1 & \alpha \sin{\left(t \right)}\\
\hline
-1 & 0 & 0 & 0 & 0 & 0 & 0\\
0 & -1 & 0 & 0 & 0 & 0 & 0\\
0 & 0 & -1 & 0 & 0 & 0 & 0\\
\hline
0 & 0 & 0 & \alpha \cos^{2}{\left(t \right)} & \alpha \sin{\left(t \right)} \cos{\left(t \right)} & - \alpha \sin{\left(t \right)} & 0\end{array}\right].
\]
\subsection{Transformation into a PreSCF}
The goal of this step is to construct $\{E_0, F_0 \}$ that are in a PreSCF and fulfill
\[
 \{H_E, H_F\}  \sim \{E_0,F_0\} 
\]
according to Section \ref{sec:Step0}. To this end, we consider a particular basis.
\begin{itemize}
	\item Since $H_{32}=\begin{bmatrix}
		\alpha \cos^{2}{\left(t \right)} & \alpha \sin{\left(t \right)} \cos{\left(t \right)} & - \alpha \sin{\left(t \right)} 
	\end{bmatrix}$, we may choose
	\[
	B_{d3} = \begin{bmatrix}
		\sin{\left(t \right)} \cos{\left(t \right)} & \sin{\left(t \right)}\\
		\sin^{2}{\left(t \right)} & - \cos{\left(t \right)}\\
		\cos{\left(t \right)} & 0
	\end{bmatrix}, \quad B_{a3}=\begin{bmatrix}
		\cos^{2}{\left(t \right)} \\ \sin{\left(t \right)} \cos{\left(t \right)} \\ -  \sin{\left(t \right)} 
	\end{bmatrix} .
	\]
	Note that $B_{d3}$ is not uniquely determined and could have been chosen in a different way.
	\item Since $H_{32}H_{21}=-H_{32}$ we can choose $B_{d2}=B_{d3}$ and $B_{a2}=B_{a3}$.
	\item Since $H_{32}H_{21}H_{13}=-\alpha^2 \neq 0$, $B_{a1}=1$.
\end{itemize}

Consider
\[ 
L_B=
\left[\begin{array}{ccc|ccc|c}\sin{\left(t \right)} \cos{\left(t \right)} & \sin^{2}{\left(t \right)} & \cos{\left(t \right)} & 0 & 0 & 0 & 0\\
\sin{\left(t \right)} & - \cos{\left(t \right)} & 0 & 0 & 0 & 0 & 0\\
\hline 
0 & 0 & 0 & \sin{\left(t \right)} \cos{\left(t \right)} & \sin^{2}{\left(t \right)} & \cos{\left(t \right)} & 0\\0 & 0 & 0 & \sin{\left(t \right)} & - \cos{\left(t \right)} & 0 & 0\\
\hline
\cos^{2}{\left(t \right)} & \sin{\left(t \right)} \cos{\left(t \right)} & - \sin{\left(t \right)} & 0 & 0 & 0 & 0\\
\hline
0 & 0 & 0 & \cos^{2}{\left(t \right)} & \sin{\left(t \right)} \cos{\left(t \right)} & - \sin{\left(t \right)} & 0\\
\hline
0 & 0 & 0 & 0 & 0 & 0 & 1\end{array}\right],
\]

\[
K_B=
\left[\begin{array}{cc|cc|c|c|c}
\sin{\left(t \right)} \cos{\left(t \right)} & \sin{\left(t \right)} & 0 & 0 & 0 & \cos^{2}{\left(t \right)} & 0\\
\sin^{2}{\left(t \right)} & - \cos{\left(t \right)} & 0 & 0 & 0 & \sin{\left(t \right)} \cos{\left(t \right)} & 0\\
\cos{\left(t \right)} & 0 & 0 & 0 & 0 & - \sin{\left(t \right)} & 0\\
\hline
0 & 0 & \sin{\left(t \right)} \cos{\left(t \right)} & \sin{\left(t \right)} & 0 & 0 & \cos^{2}{\left(t \right)}\\
0 & 0 & \sin^{2}{\left(t \right)} & - \cos{\left(t \right)} & 0 & 0 & \sin{\left(t \right)} \cos{\left(t \right)}\\
0 & 0 & \cos{\left(t \right)} & 0 & 0 & 0 & - \sin{\left(t \right)}\\
\hline
0 & 0 & 0 & 0 & 1 & 0 & 0\end{array}\right].
\]
Step 0:  Wit $L_0=L_BL_P$, $K_0=K_PK_B$ we obtain
\[
E_0=
\left[\begin{array}{cccc|ccc}1 & 0 & 0 & 0 & 0 & 0 & 0\\0 & 1 & 0 & 0 & 0 & 0 & 0\\0 & 0 & 1 & 0 & 0 & 0 & 0\\0 & 0 & 0 & 1 & 0 & 0 & 0\\
\hline 0 & 0 & 0 & 0 & 0 & 1 & 0\\0 & 0 & 0 & 0 & 0 & 0 & 1\\0 & 0 & 0 & 0 & 0 & 0 & 0\end{array}\right]
,\]
\[
F_0=
\left[\begin{array}{cccc|ccc}
- \frac{\sin{\left(2 t \right)}}{2} & 0 & \cos^{2}{\left(t \right)} & 0 & 0 & - \cos^{2}{\left(t \right)} & - \frac{\sin{\left(2 t \right)}}{2}\\0 & 0 & \cos{\left(t \right)} & 0 & 0 & 0 & - \sin{\left(t \right)}\\-1 & 0 & 0 & \sin{\left(t \right)} & 0 & 0 & -1\\0 & -1 & - \sin{\left(t \right)} & 0 & 0 & 0 & - \cos{\left(t \right)}\\ 
\hline
\sin^{2}{\left(t \right)} & 0 & - \frac{\sin{\left(2 t \right)}}{2} & 0 & - \alpha & \frac{\sin{\left(2 t \right)}}{2} & \sin^{2}{\left(t \right)}\\0 & 0 & 1 & \cos{\left(t \right)} & 0 & -1 & 0\\0 & 0 & 0 & 0 & 0 & 0 & \alpha
\end{array}\right].
\]
\subsection{Steps 1 and 2: Iterative procedures}

Step 1: For
\[
L_1=
\left[\begin{array}{cccc|ccc}
1 & 0 & 0 & 0 & 0 & - \cos^{2}{\left(t \right)} & \frac{\left(\cos^{2}{\left(t \right)} + 3\right) \sin{\left(2 t \right)}}{2 \alpha}\\0 & 1 & 0 & 0 & 0 & 0 & \frac{\sin{\left(t \right)}}{\alpha}\\0 & 0 & 1 & 0 & 0 & 0 & \frac{\cos^{2}{\left(t \right)} + 1}{\alpha}\\ 
0 & 0 & 0 & 1 & 0 & 0 & \frac{\cos{\left(t \right)}}{\alpha}\\
\hline
0 & 0 & 0 & 0 & 1 & 0 & 0\\0 & 0 & 0 & 0 & 0 & 1 & 0\\0 & 0 & 0 & 0 & 0 & 0 & 1 \end{array}\right]
,\quad 
K_1=
\left[\begin{array}{cccc|ccc}1 & 0 & 0 & 0 & 0 & 0 & \cos^{2}{\left(t \right)}\\0 & 1 & 0 & 0 & 0 & 0 & 0\\0 & 0 & 1 & 0 & 0 & 0 & 0\\0 & 0 & 0 & 1 & 0 & 0 & 0\\
\hline
0 & 0 & 0 & 0 & 1 & 0 & 0\\0 & 0 & 0 & 0 & 0 & 1 & 0\\0 & 0 & 0 & 0 & 0 & 0 & 1\end{array}\right],
\]
by Lemma \ref{lem:M12} we obtain 
$
E_1=E_0
$
and
\[
F_1=
\left[\begin{array}{cccc|ccc}
- \frac{\sin{\left(2 t \right)}}{2} & 0 & 0 & - \cos^{3}{\left(t \right)} & 0 & 0 & 0\\0 & 0 & \cos{\left(t \right)} & 0 & 0 & 0 & 0\\-1 & 0 & 0 & \sin{\left(t \right)} & 0 & 0 & 0\\0 & -1 & - \sin{\left(t \right)} & 0 & 0 & 0 & 0\\
\hline
\sin^{2}{\left(t \right)} & 0 & - \frac{\sin{\left(2 t \right)}}{2} & 0 & - \alpha & \frac{\sin{\left(2 t \right)}}{2} & \left(\cos^{2}{\left(t \right)} + 1\right) \sin^{2}{\left(t \right)}\\0 & 0 & 1 & \cos{\left(t \right)} & 0 & -1 & 0\\0 & 0 & 0 & 0 & 0 & 0 & \alpha
\end{array}\right].
\]

\medskip
Step 2: For 
\[
L_2=
\left[\begin{array}{cccc|ccc} 
1 & 0 & 0 & 0 & 0 & 0 & 0\\0 & 1 & 0 & 0 & 0 & 0 & 0\\0 & 0 & 1 & 0 & 0 & 0 & 0\\0 & 0 & 0 & 1 & 0 & 0 & 0\\
\hline0 & 0 & -1 & - \cos{\left(t \right)} & 1 & 0 & 0\\0 & 0 & 0 & 0 & 0 & 1 & 0\\0 & 0 & 0 & 0 & 0 & 0 & 1 \end{array}
\right]
,\]
\[ 
K_2=
\left[\begin{array}{cccc|ccc}
1 & 0 & 0 & 0 & 0 & 0 & 0\\0 & 1 & 0 & 0 & 0 & 0 & 0\\0 & 0 & 1 & 0 & 0 & 0 & 0\\0 & 0 & 0 & 1 & 0 & 0 & 0\\
\hline
\frac{\sin^{2}{\left(t \right)} + 1}{\alpha} & \frac{\cos{\left(t \right)}}{\alpha} & \frac{\sin{\left(2 t \right)}}{2 \alpha} & - \frac{\sin^{3}{\left(t \right)} + \sin{\left(t \right)}}{\alpha} & 1 & 0 & 0\\0 & 0 & 1 & \cos{\left(t \right)} & 0 & 1 & 0\\0 & 0 & 0 & 0 & 0 & 0 & 1
\end{array}\right],
\]
$E_2=E_1=E_0$ and
by Lemma \ref{lem:M21} we obtain 
$
E_2=E_1=E_0
$
and
\[
F_2=
\left[\begin{array}{cccc|ccc}
- \frac{\sin{\left(2 t \right)}}{2} & 0 & 0 & - \cos^{3}{\left(t \right)} & 0 & 0 & 0\\0 & 0 & \cos{\left(t \right)} & 0 & 0 & 0 & 0\\-1 & 0 & 0 & \sin{\left(t \right)} & 0 & 0 & 0\\0 & -1 & - \sin{\left(t \right)} & 0 & 0 & 0 & 0\\
\hline 
0 & 0 & 0 & 0 & - \alpha & \frac{\sin{\left(2 t \right)}}{2} & \left(\cos^{2}{\left(t \right)} + 1\right) \sin^{2}{\left(t \right)}\\0 & 0 & 0 & 0 & 0 & -1 & 0\\0 & 0 & 0 & 0 & 0 & 0 & \alpha
\end{array}\right].
\]

\subsection{Representation of the canonical projector}
From the representation of $K_1$ and $K_2$ we know
\[
A=\begin{bmatrix}
	 0 & 0 &  \cos^{2}{\left(t \right)}\\
 0 & 0 & 0\\
 0 & 0 & 0\\
 0 & 0 & 0\\
\end{bmatrix}, 
\]
\[ 
B=\begin{bmatrix}
\frac{\sin^{2}{\left(t \right)} + 1}{\alpha} & \frac{\cos{\left(t \right)}}{\alpha} & \frac{\sin{\left(2 t \right)}}{2 \alpha} & - \frac{\sin^{3}{\left(t \right)} + \sin{\left(t \right)}}{\alpha} \\0 & 0 & 1 & \cos{\left(t \right)} \\0 & 0 & 0 & 0 
\end{bmatrix}.
\]
Also in this case,  $AB=0$, such that

\[
\pPi_{can}=K_0
\left[\begin{array}{cccc|ccc}1 & 0 & 0 & 0 & 0 & 0 & - \cos^{2}{\left(t \right)}\\0 & 1 & 0 & 0 & 0 & 0 & 0\\0 & 0 & 1 & 0 & 0 & 0 & 0\\0 & 0 & 0 & 1 & 0 & 0 & 0\\
\hline
\frac{\sin^{2}{\left(t \right)} + 1}{\alpha} & \frac{\cos{\left(t \right)}}{\alpha} & \frac{\sin{\left(2 t \right)}}{2 \alpha} & - \frac{\sin^{3}{\left(t \right)} + \sin{\left(t \right)}}{\alpha} & 0 & 0 & \frac{\left(\cos^{2}{\left(t \right)} - 2\right) \cos^{2}{\left(t \right)}}{\alpha}\\0 & 0 & 1 & \cos{\left(t \right)} & 0 & 0 & 0\\
0 & 0 & 0 & 0 & 0 & 0 & 0
\end{array}\right]K_0^{-1}.
\]

Note that for the considered index-2 and index-3 DAEs in Hessenberg form and orthogonal $K_0$, $AB=0$ follows by construction, such that in these cases we always have
\[
\pPi_{can}= K_0 \begin{bmatrix}
	I_d  & -A \\
	B & -BA
\end{bmatrix}K_0^{-1}.
\]

\section{Conclusion}

DAEs in SCF or even SSCF are easy to understand and to handle.
However, the equivalence transformation of linear time varying DAEs into an SCF is usually rather difficult. 
Fortunately, for some classes of structured DAEs with known canonical characteristics, a transformation into the preliminary stage PreSCF introduced in Definition \ref{def:prel_SCF} results to be simpler and as soon as we have it, the computation of the SCF  becomes possible with the described steps that take advantage of nilpotency.
\medskip

Of course, in general the computation of this preliminary stage is not trivial either, but, at least for some classes of DAEs, relatively easy. 
We showed how to transform the  T-canonical and the S-canonical forms from the tractability and strangeness frameworks into PreSCF. Moreover, we  considered Hessenberg forms that have a completely different structure and computed the SCF step-by-step via the PreSCF.
We are confident that  for further classes of structured DAEs appropriate general descriptions of suitable transformations may be found.
\medskip

Once the SCF and the corresponding transformation matrix function $K$ that describes the transformation of unknowns is found, a description of the canonical subspaces $S_{can}$ and $N_{can}$, as well as the canonical projector is possible. This can be used in particular to formulate accurate initial conditions and transfer conditions from one
 time-window to the next in a time stepping procedure, cf.\ \cite{hanke_maerz2025}. 
In particular, we provide this for DAEs arising from linear or linearized multibody systems.

\appendix

\section{Appendix}

On the one hand, in this appendix we present a self-contained compilation of technical details concerning block upper triangular (BUT) matrix functions in a self-contained form.
Unlike the appendices of \cite{commonground2024} and \cite{SSCF25}, our focus is not limited solely to strictly block upper triangular (SUT) and hence nilpotent matrices, although these are included in our discussion as well.
\medskip

On the other hand, we present two basic equivalence transformations that were our starting point for the development of the algorithm from Section \ref{sec:Procedure}.
\medskip

We close the appendix with a remark on time-varying transformations of explicit ODE and its consequences for pure ODEs of DAEs.

\subsection{Block structured upper triangular matrix functions}\label{Appendix_block}
Suppose that for $\mu \in \N, \mu\geq 2$ some   integers
\[
\kappa_0 \geq  \theta_0\geq\cdots\geq\theta_{\mu-2}>0
\]
are given and  consider
\begin{itemize}
	\item either decreasing $\ell_1\geq\cdots\geq\ell_{\mu}>0$:
	\[
	\ell_1= \kappa_0, \quad \ell_i
= \theta_{i-2}, \quad i=2, \ldots, \mu,
	\]
	\item or increasing  $0<\ell_1\leq\cdots\leq\ell_{\mu}$:
	\[
	 \ell_i= \theta_{\mu-i-1}, \quad i=1, \ldots, \mu-1, \quad \ell_{\mu}=\kappa_0,
	\]
	\end{itemize}
 and $\ell := \sum_{i=1}^{\mu} \ell_i =\kappa_0+\sum_{j=0}^{\mu-2} \theta_j $. These two different orders appear in the well established canonical forms discussed in Section \ref{sec:T-S-canonical}, that have several common properties that we summarize here. 
\medskip

For $\mathcal I \subseteq \reals$ and fixed $\ell, \mu, \ell_1, \ldots, \ell_{\mu}$ we denote by 
\begin{itemize}
	\item $\BUT =\BUT(\ell,\mu,\ell_1,\ldots,\ell_{\mu})$ the set of all block upper triangular matrix functions
 $B :\mathcal I\rightarrow \reals^{\ell \times \ell}$
\[
B =
	\begin{bmatrix}
   B_{11}&B_{12}&&\cdots&B_{1,\mu} \\
   0&B_{22}&B_{23}& &0\\
   \vdots &\ddots &\ddots&\ddots&\vdots\\
   0&&&B_{\mu-1, \mu}&B_{\mu-1, \mu}\\
   0&0&\cdots&0&B_{\mu, \mu}
   \end{bmatrix}\begin{array}{lll}
		\left.\right\}  ~\ell_1 \\
		\left.\right\}  ~\ell_2\\
		\quad  \vdots &   \\
		\left.\right\} 	\ell_{\mu-1} \\
		\left.\right\} 	~\ell_{\mu}
	 \end{array}     
\]
with blocks
\begin{equation*}
B_{i,j} :\mathcal I\rightarrow \reals^{\ell_{i} \times \ell_{j}}, \quad \text{and} \quad B_{i,j} =0 \quad \text{for} \quad i>j. 
\end{equation*}
\item $\BUT_{nonsingular} =\BUT_{nonsingular}(\ell,\mu,\ell_1,\ldots,\ell_{\mu})$ the set of all nonsingular $R\in \BUT$ that therefore have nonsingular diagonal blocks
\begin{equation*}
R_{i,i} :\mathcal I\rightarrow \reals^{\ell_{i} \times \ell_{i}} \quad \text{for all} \quad t \in \mathcal I. 
\end{equation*}
\item $\SUT =\SUT(\ell,\mu,\ell_1,\ldots,\ell_{\mu})$  the set of all strictly block upper triangular matrix functions $N \in \BUT$, i.e., with
\[
N_{i,j} =0 \quad \text{for} \quad i\geq j.
\]

\item  In case of $\ell_1\geq \cdots\geq \ell_{\mu}$ we denote by 
\begin{itemize}
	\item 
$\SUT_{column}\subset \SUT$ the set  of all $N\in \SUT$ having exclusively blocks $(N)_{i,i+1}$ with full column rank, that is 
\begin{align}\label{N.col}
 \rank (N)_{i,i+1}=\ell_{i+1},\quad i=1,\ldots,\mu-1\quad \text{for all} \quad t \in \mathcal I.
\end{align}

\item  $N^{(E_c)} \in \SUT_{column}$ the elementary matrix
\[
N^{(E_c)} :=
	\begin{bmatrix}
   0&N^{(E_c)}_{12}&0&\cdots&0 \\
   &0&N^{(E_c)}_{23}& 0&0\\
   &&\ddots&\ddots&\vdots\\
   &&&&N^{(E_c)}_{\mu-1, \mu}\\
   &&&&0
   \end{bmatrix}\begin{array}{lll}
		\left.\right\}  ~\ell_1&=&\kappa_0 \\
		\left.\right\}  ~\ell_2&=&\theta_0 \\
		\quad  \vdots & &\vdots  \\
		\left.\right\} 	\ell_{\mu-1}&=&\theta_{\mu-3} \\
		\left.\right\} 	~\ell_{\mu}&=&\theta_{\mu-2}
	 \end{array}     
\]
with blocks
\begin{equation*}
N^{(E_c)}_{i,i+1} = \begin{bmatrix}
	I_{\ell_{i+1}} \\
	0
\end{bmatrix} \in \reals^{\ell_{i} \times \ell_{i+1}}. 
\end{equation*}
\end{itemize}
\item In case of $\ell_1\leq \cdots\leq \ell_{\mu}$ we denote by
\begin{itemize}
	\item  $\SUT_{row}\subset \SUT$ the set  of all $N\in \SUT$ having exclusively blocks $(N)_{i,i+1}$ with full row rank, that is 
\begin{align}\label{N.row}
 \rank (N)_{i,i+1}=\ell_{i},\quad i=1,\ldots,\mu-1 \quad \text{for all} \quad t \in \mathcal I.
\end{align}
\item $N^{(E_r)}  \in \SUT_{row}$ the elementary matrix
\[
N^{(E_r)} =
	\begin{bmatrix}
   0&N^{(E_r)}_{12}&0&\cdots&0 \\
   &0&N^{(E_r)}_{23}& 0&0\\
   &&\ddots&\ddots&\vdots\\
   &&&&N^{(E_r)}_{\mu-1, \mu}\\
   &&&&0
   \end{bmatrix} \begin{array}{lll}
		\left.\right\}  ~\ell_1&=& \theta_{\mu-2} \\
		\left.\right\}  ~\ell_2&=&\theta_{\mu-3} \\
		\quad  \vdots & &\vdots  \\
		\left.\right\} 	\ell_{\mu-1}&=&\theta_0 \\
		\left.\right\} 	~\ell_{\mu}&=&\kappa_0
	 \end{array}  
\]
with blocks
\begin{equation*}
N^{(E_r)}_{i,i+1} = \begin{bmatrix}
	0 & I_{\ell_{i}} 
\end{bmatrix} \in \reals^{\ell_{i} \times \ell_{i+1}}. 
\end{equation*}
\end{itemize}
\end{itemize}

We refer to \cite{commonground2024}, \cite{SSCF25} for some useful properties of $\SUT$, $\SUT_{column}$, $\SUT_{row}$, $N^{(E_c)}$ and $N^{(E_r)}$ and use also the notations
\[
\SUT_{c/r}, \quad   N^{c/r}, \quad N^{(E_{c/r})}
\]
if $\SUT_{column}, N^{c}, N^{(E_{c})}$ is meant for decreasing $\ell_i$ and  $\SUT_{row}, N^{r}, N^{(E_{r})}$ for increasing $\ell_i$, respectively. 

With this notation, all $N^{c/r} \in \SUT_{c/r}$ are nilpotent of order $\mu$ and for all $ t \in \mathcal I$
	\[
	\rank N^{c/r} = \ell-\kappa_0, \quad \theta_i=\rank (N^{c/r})^{i+1}-\rank (N^{c/r})^{i+2}, \quad i=0, \ldots, \mu-2.
	\]

\begin{lemma}\label{lem:R_BUT}
For $R \in \BUT_{nonsingular}$ it holds
\begin{itemize}
  \item $R^{-1} \in \BUT_{nonsingular}$, $(R^{-1})_{ii}=(R_{ii})^{-1}$,
	\item $N^{(E_{c/r})}R \in \SUT_{c/r}$,
	\item $R N^{(E_{c/r})} \in \SUT_{c/r}$,
	\item For any $N \in \SUT$, $R+N \in \BUT_{nonsingular}$.
\end{itemize}
\end{lemma}
\begin{proof}
Straight forward computation using the nonsingularity of the diagonal blocks $R_{ii}$.
\end{proof}

\begin{lemma}\cite{SSCF25}\label{lem:Ncr}
\begin{itemize}
\item Every pair $\{N^c,I\}$ with a sufficiently smooth matrix function $N^c \in \SUT_{column}$ can be equivalently transformed into a pair $\{\hat{N}^c,I\}$ with a matrix function $\hat{N}^{c}\in \SUT_{column}$ showing special secondary diagonal blocks
\[
	\hat{N}^c_{i,i+1}=\begin{bmatrix}
		R_{i+1,i+1}\\
		0
	\end{bmatrix},
	\]
	and pointwise nonsingular $R_{i+1,i+1}$.
\item Every pair $\{N^r,I\}$ with a sufficiently smooth matrix function $N^r \in \SUT_{row}$ can be equivalently transformed into a  pair $\{\hat{N}^r,I\}$ with a matrix function $\hat{N}^{r}\in \SUT_{row}$ showing special secondary diagonal blocks
\[
	\hat{N}^r_{i,i+1}=\begin{bmatrix}
		R_{i,i}& 0
	\end{bmatrix}
	\]
	and pointwise nonsingular $R_{i,i}$.
\end{itemize}
\end{lemma}
\begin{proof}
A proof that uses the SVD can be found in \cite{SSCF25}. A visualization of the pattern  is given in Figure \ref{fig:Bilder_RcRr}.
\end{proof}

\begin{lemma}\cite{SSCF25} \label{lem:Rcr}
\begin{itemize}
	\item In case of decreasing $\ell_1\geq \cdots\geq \ell_{\mu}$,  for any $N^c \in \SUT_{column}$ it holds that
	\[
	R^c:=N^c(N^{(E_c)})^T + (I-N^{(E_c)}(N^{(E_c)})^T) \in \BUT_{nonsingular} 
	\]
	and 
	\[
	R^c(N^{(E_c)})=N^c.
	\]	
	Moreover, if the secondary diagonal blocks have the pattern 
	\[
	N^c_{i,i+1}=\begin{bmatrix}
		R_{i+1,i+1}\\
		0
	\end{bmatrix},
	\]
	with pointwise nonsingular $R_{i+1,i+1}$, then $R^c$ is pointwise  nonsingular 
	and $(R^c)^{-1}N^c=N^{(E_c)}$.
	\item In case of  increasing $\ell_1\leq \cdots\leq \ell_{\mu}$, for any $N^r \in \SUT_{row}$ it holds that
	\[
	R^r:=(N^{(E_r)})^TN^r + (I-(N^{(E_r)})^T N^{(E_r)}) \in \BUT_{nonsingular}
	\]
	and 
	\[
	(N^{(E_r)})R^r=N^r.
	\]
	Moreover, if the secondary diagonal blocks have the pattern 
	\[
	N^r_{i,i+1}=\begin{bmatrix}
		R_{i+1,i+1}& 
		0
	\end{bmatrix},
	\]
	with pointwise nonsingular $R_{i+1,i+1}$, then $R^r$ is nonsingular and
	$N^r(R^{r})^{-1} =N^{(E_r)}$.
\end{itemize}
\end{lemma}
\begin{proof}
A proof can be found in \cite{SSCF25} (Lemma 4.2 and Lemma 5.2) therein, and the reasoning thereafter. A visualization of the pattern of the nonsingular matrix functions is given in Figure \ref{fig:Bilder_RcRr}. 
\end{proof}

\begin{corollary}\label{cor:Rcr}
For every $N\in \SUT_{c/r}$ there exists a $R^{c/r}\in \BUT_{nonsingular}$ such that
\[
\{ N, I\} \sim \{ N^{(E_{c/r})},  (R^{c/r})^{-1} \}.
\]
\end{corollary}
\begin{proof}
The assertions follows from Lemmas \ref{lem:Ncr} and \ref{lem:Rcr}.
\end{proof}

\begin{figure}
\includegraphics[width=\textwidth]{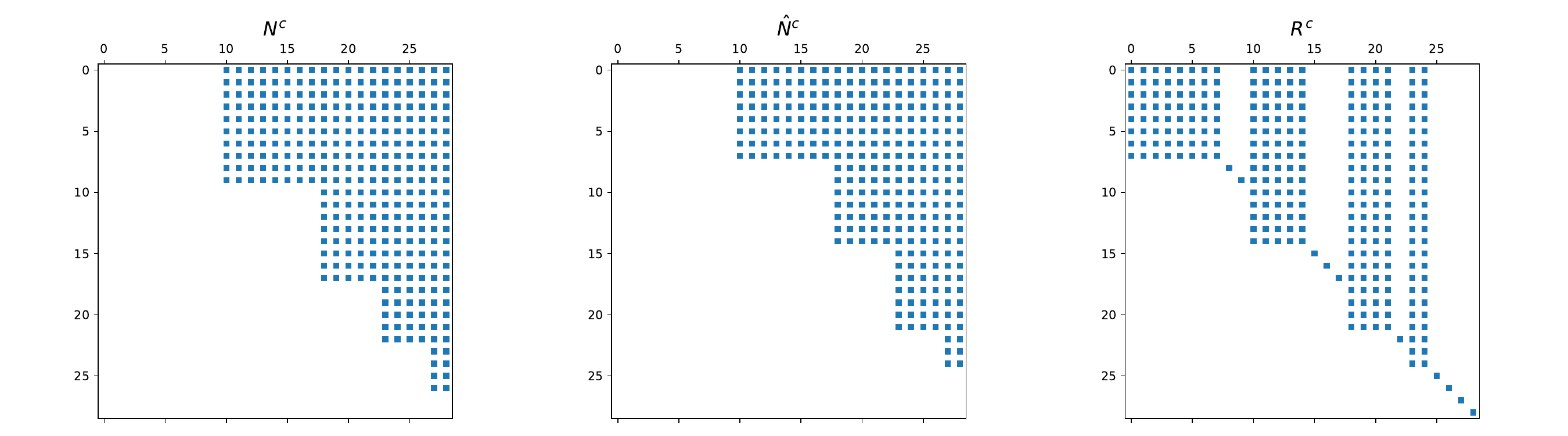}
\includegraphics[width=\textwidth]{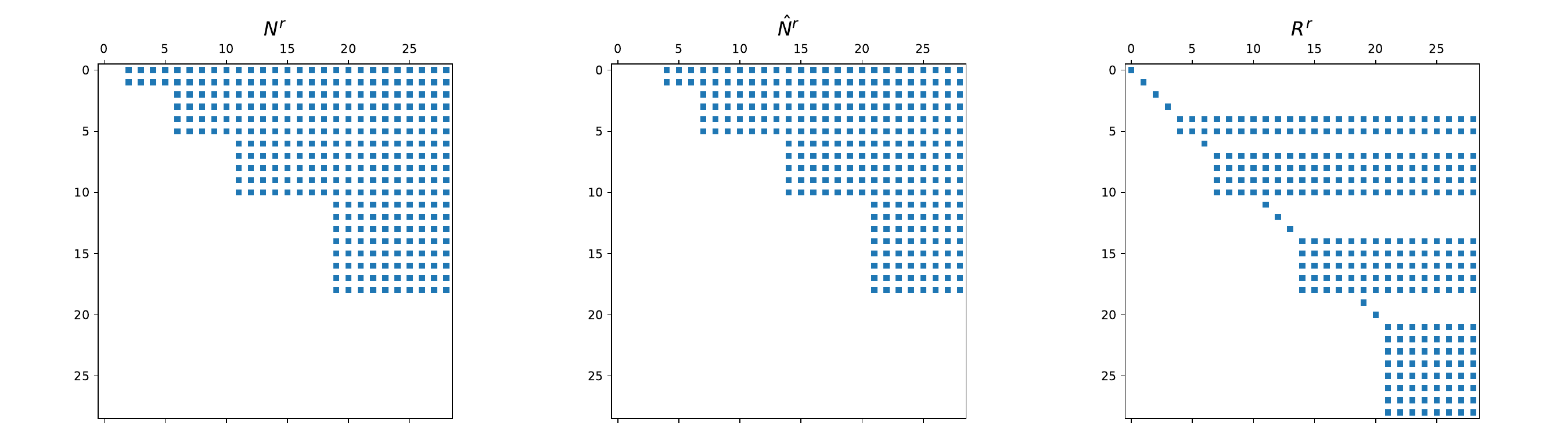}
\caption{Visualization of the pattern of $N^c$, $\hat{N}^c$ and a nonsingular $R^c$ constructed with $\hat{N}^c$ for decreasing block sizes (top) as well as $N^r$, $\hat{N}^r$ and a corresponding  nonsingular $R^r$ for increasing block sizes (bottom), cf.\ Lemma \ref{lem:Ncr} and Lemma \ref{lem:Rcr}.}
\label{fig:Bilder_RcRr}
\end{figure}

\subsection{Two elementary equivalence transformations} \label{Appendix:Equivalence}
\begin{lemma}\label{lem:M12}
For every pair in the form
\[
\{E,F\}=\left\{ \begin{bmatrix}
	I_d & 0 \\
	0 & E_{22}
\end{bmatrix}, \begin{bmatrix}
	F_{11}& F_{12} \\
	F_{21} & F_{22}
\end{bmatrix} \right\}, 
\]
every matrix function $M_{12} :\mathcal I\rightarrow \reals^{d \times a}$ and transformation matrix functions
\[
L=\begin{bmatrix}
	I_{d} & M_{12} \\
	0 & I_{a}
\end{bmatrix}, \quad 
K=\begin{bmatrix}
	I_{d} & -M_{12}E_{22} \\
	0 & I_{a}
\end{bmatrix},
\]
it holds
\begin{eqnarray*}
LEK&=&E, \\ 
LFK+LEK'&=&\begin{bmatrix}
F_{11} + M_{12}F_{21}  & F_{12} + M_{12}F_{22}  -  \left(F_{11} + M_{12}F_{21} \right) M_{12} E_{22}\\F_{21} &   F_{22} - F_{21} M_{12}E_{22} 
\end{bmatrix}\\
&& +\begin{bmatrix}
	0 & -M_{12}'E_{22} \\
	0 & 0
\end{bmatrix}.
\end{eqnarray*}

\end{lemma}
\begin{proof}
Straight forward computation gives the result.
\end{proof}
For nonsingular $F_{22}$ the choice $M_{12} = -F_{12} F_{22}^{-1}$ is
of particular interest, cf.\ Lemma \ref{lem:M12F22}.

\begin{lemma}\label{lem:M21}
For every  pair in the form
\[
\{E,F\}=\left\{ \begin{bmatrix}
	I_d & 0 \\
	0 & E_{22}
\end{bmatrix}, \begin{bmatrix}
	F_{11}& F_{12} \\
	F_{21} & F_{22}
\end{bmatrix} \right\}, 
\]
every matrix function $M_{21} :\mathcal I\rightarrow \reals^{a \times d} $ and transformation matrix functions
\[
L=\begin{bmatrix}
	I_{d} & 0 \\
	E_{22}M_{21} & I_{a}
\end{bmatrix}, \quad 
K=\begin{bmatrix}
	I_{d} & 0 \\
	-M_{21} & I_{a}
\end{bmatrix},
\]
it holds
\begin{eqnarray*}
LEK&=&E, \\
LFK+LEK'&=&\begin{bmatrix}
	F_{11} - F_{12} M_{21}& F_{12} \\
	F_{21} -F_{22}M_{21}+E_{22}M_{21} \left(F_{11}    -  F_{12} M_{21}\right) & F_{22}+E_{22}M_{21}F_{12}
\end{bmatrix}
\\
&& +\begin{bmatrix}
	0 & 0 \\
	-E_{22}M_{21}' & 0
\end{bmatrix}.
\end{eqnarray*}
\end{lemma}

\begin{proof}
Straight forward computation gives the result.
\end{proof}
In this case, for nonsingular $F_{22}$ the choice $M_{21} = F_{22}^{-1}F_{21}$ is
of particular interest, cf.\ Lemma \ref{lem:M21F22}.
\medskip

\subsection{A note on transformations of explicit (pure) ODEs} \label{appendix:PureODEs}

Consider an explicit ODE 
\[
x'+\Omega x=0, 
\]
with a time-varying matrix function $\Omega: \mathcal I \rightarrow \reals^{d\times d}$, and
an unknown function $x$. 
A transformation $x=K\tilde{x}$ by a pointwise nonsingular matrix function $K$  leads to 
\[
x'+\Omega x=K'\tilde{x}+K\tilde{x}'+\Omega K\tilde{x}=K\tilde{x}'+(\Omega K+K')\tilde{x},
\]
such that the premultiplication by $L:=K^{-1}$ provides
the transformed ODE $\tilde{x}'+\tilde{\Omega} \tilde{x}=0$ with  $\tilde{\Omega}=K^{-1}(\Omega K+K')$.
\medskip

Let us now suppose that for any given $\alpha \in \reals$ we want to compute $K$ such that $\tilde{\Omega}=\alpha I$. From $\alpha I=K^{-1}(\Omega K+K')$ it follows that such a matrix function $K$ has to fulfill the system of ODEs
\[
K'=(\alpha I-\Omega)K.
\]

If we choose an nonsingular initial value $K(0)=K_0$, the solution $K$ is nonsingular on $\mathcal I$ due to the formula of Liouville. Indeed, $K$ is the fundamental matrix of the above ODE. Consequently,
\[
\left\{ I, \Omega \right\} \overset{K^{-1}, K}{\sim} \left\{ I, \alpha I \right\}.
\]

Notice that for $\alpha_1>0$ and $\alpha_2<0$ we therefore can find  nonsingular $K_1$ and  $K_2$, such that the equivalent ODEs present completely different stability properties. This is a crucial difference between constant transformations of coordinates and time-varying transformations.
\medskip

This means that for DAEs, as showed also in Examples \ref{ex:ODEK11-1} and \ref{ex:ODEK11-2}, different pure ODEs may present completely different stability properties.
However, in both Examples, either $\left|K_{11}\right|$ or $\left|K_{11}^{-1}\right|$
grows unboundedly for $t\rightarrow \infty$. 
\medskip

Moreover, the above considerations imply that with the scope of the here permitted transformations we can formulate a further possible definition of regular DAEs.

\begin{proposition}
 A pair of matrix functions $\{E,F\}$ is regular, iff  
\[
\{E,F\} \ \sim \  \left\{ \begin{bmatrix}
	I_{d} & 0 \\
	0 & N^{(E_{c/r})}
\end{bmatrix}, \begin{bmatrix}
	\alpha I_d & 0 \\
	0 & I_{a}
\end{bmatrix} \right\},
\]
for any $\alpha \in \reals$.
\end{proposition}

\bibliography{Computing-SCF}{}

\begin{thebibliography}{10}

\bibitem{AscherPetzold91}
U.~M. Ascher and L.~R. Petzold.
\newblock Projected implicit {Runge}-{Kutta} methods for differential-algebraic
  equations.
\newblock {\em SIAM J. Numer. Anal.}, 28(4):1097--1120, 1991.

\bibitem{BergerIlchmann}
T.~Berger and A.~Ilchmann.
\newblock On the standard canonical form of time-varying linear {DAE}s.
\newblock {\em Quarterly of Applied Mathematic}, LXXI(1):69--87, 2013.

\bibitem{CampbellMoore95}
S.~L. Campbell and E.~Moore.
\newblock Constraint preserving integrators for general nonlinear higher index
  {DAEs}.
\newblock {\em Numer. Math.}, 69(4):383--399, 1995.

\bibitem{CaPe83}
S.~L. Campbell and L.~R. Petzold.
\newblock Canonical forms and solvable singular systems of differential
  equations.
\newblock {\em SIAM Journal on Algebraic Discrete Methods}, 4(4):517--521,
  1983.

\bibitem{Eich-SoelFueh98}
E.~Eich-Soellner and C.~F{\"u}hrer.
\newblock {\em Numerical Methods in Multibody Dynamics}.
\newblock European Consortium for Mathematics in Industry (ECMI series).
  Teubner, 1998.

\bibitem{HaMaMae98}
M.~Hanke, E.~Izquierdo~Macana, and R.~M{\"a}rz.
\newblock On asymptotics in case of linear index-2 differential-algebraic
  equations.
\newblock {\em SIAM J. Numer. Anal.}, 35(4):1326--1346, 1998.

\bibitem{HaMae2023}
M.~Hanke and R.~M{\"a}rz.
\newblock Canonical subspaces of linear time-varying differential-algebraic
  equations and their usefulness for formulating accurate initial conditions.
\newblock {\em DAE Panel}, 1:1--32, 2023.

\bibitem{hanke_maerz2025}
M.~Hanke and R.~März.
\newblock On the computation of accurate initial conditions for linear
  higher-index differential-algebraic equations and its application in initial
  value solvers.
\newblock {\em Numerical Algorithms}, 2025.

\bibitem{KuMe2024}
P.~Kunkel and V.~Mehrmann.
\newblock {\em Differential-algebraic equations. {Analysis} and numerical
  solution}.
\newblock EMS Textb. Math. Berlin: European Mathematical Society (EMS), 2nd
  edition edition, 2024.

\bibitem{CRR}
R.~Lamour, R.~M{\"a}rz, and C.~Tischendorf.
\newblock {\em Differential-Algebraic Equations: A Projector Based Analysis}.
\newblock Differential-Algebraic Equations Forum. Springer-Verlag Berlin
  Heidelberg New York Dordrecht London, 2013.
\newblock eds. Achim Ilchmann and Timo Reis.

\bibitem{LinhMaerz}
V.~H. Linh and R.~M{\"a}rz.
\newblock Adjoint pairs of differential-algebraic equations and their
  {Lyapunov} exponents.
\newblock {\em J. Dyn. Differ. Equations}, 29(2):655--684, 2017.

\bibitem{commonground2024}
D.~Est{\'e}vez Schwarz, R.~Lamour, and R.~März.
\newblock The common ground of {DAE} approaches. {An} overview of diverse {DAE}
  frameworks emphasizing their commonalities, 2024.
\newblock To appear in DAE Panel.

\bibitem{SSCF25}
D.~Est{\'e}vez Schwarz, R.~Lamour, and R.~März.
\newblock Regular linear time varying {DAE}s are equivalent to {DAE}s in strong
  standard canonical form, 2025.
\newblock Submitted.

\end{thebibliography}
\bibliographystyle{plain}


\end{document}